\documentclass[11pt, reqno]{amsart}
\usepackage{amssymb,amscd,amsthm,amsxtra,amsfonts}
\usepackage{latexsym}
\usepackage{blindtext}
\usepackage{chngcntr}
\usepackage{fancyhdr}
\usepackage{color}
\usepackage{amsmath}

\usepackage[usenames,dvipsnames,svgnames,table]{xcolor}

\usepackage{graphicx}  \usepackage{epstopdf}

\usepackage{ dsfont }

\bibliography{refs}

\usepackage{mathrsfs}
\usepackage[mathscr]{euscript}
\renewcommand{\mathcal}[1]{{\mathscr#1}}

\makeatletter
\@addtoreset{section}{part}
\makeatother

\newtheorem{theorem}{Theorem}[section]

\newtheorem{lemma}[theorem]{Lemma}
\newtheorem{prop}[theorem]{Proposition}

\theoremstyle{definition}

\theoremstyle{remark}
\newtheorem{remark}[theorem]{Remark}
\numberwithin{equation}{section}

\newcommand{\R}{{\mathbb R}}

\renewcommand{\leq}{\leqslant}
\renewcommand{\le}{\leqslant}
\renewcommand{\geq}{\geqslant}
\renewcommand{\ge}{\geqslant}
\newcommand{\la}{\lambda}

\newcommand{\oxi}{\overline \xi}
\newcommand{\uxi}{\underline \xi}
\newcommand{\ozeta}{\overline \zeta}
\newcommand{\uzeta}{\underline \zeta}

\newcommand{\ophi}{\overline{\phi}}
\newcommand{\uphi}{\underline{\phi}}

\newcommand{\oophi}{\overline{\overline{\phi}}}

\newcommand{\oT}{\overline{T}}

\newcommand{\oh}{\overline{h}}

\newcommand{\oZ}{\overline{Z}}

\newcommand{\oc}{\overline{c}}

\newcommand{\uU}{\underline U}
\newcommand{\oU}{\overline U}

\renewcommand{\epsilon}{\varepsilon }

\newlength{\defbaselineskip}
\newcommand{\setlinespacing}[1]
           {\setlength{\baselineskip}{#1 \defbaselineskip}}

\begin{document}
\title[Doubling Nodal Solutions  with Maximal Rank]{Doubling Nodal Solutions to the Yamabe
Equation in $\R^n$ with maximal rank}

\author[M. Medina]{Maria Medina}
\address[Mar\'ia Medina]{Departamento de Matem\'aticas,
Universidad Aut\'onoma de Madrid,
Ciudad Universitaria de Cantoblanco,
28049 Madrid, Spain}
\email{maria.medina@uam.es}

\author[M. Musso]{Monica Musso}
\address[Monica Musso]{Department of Mathematical Sciences, University of Bath,
 Bath BA2 7AY, United Kingdom}
\email{m.musso@bath.ac.uk}

\thanks{The first author was partially supported by the European Union's Horizon $2020$ research and innovation programme under the Marie Sklodowska-Curie grant agreement N $754446$ and UGR Research and Knowledge Transfer Found - Athenea3i, and by Project PDI2019-110712GB-100, MICINN, Spain. The second author is supported by
 EPSRC Research Grant EP/T008458/1.}

\medskip

\begin{abstract}
We construct a new family of entire solutions to the Yamabe equation
$$-\Delta u=\frac{n(n-2)}{4}|u|^{\frac{4}{n-2}}u \mbox{ in }\mathcal{D}^{1,2}(\R^n).$$
If $n=3$ our solutions have maximal rank, being the first example in odd dimension. Our construction has analogies with the doubling of the equatorial spheres in the construction of  minimal surfaces in $S^3(1)$.

\end{abstract}

\maketitle


\section{Introduction}

\noindent Consider the problem
\begin{equation}\label{prob}
-\Delta u=\gamma |u|^{p-1}u\hbox{ in }\mathbb{R}^n,\qquad \gamma:=\frac{n(n-2)}{4},\qquad u\in\mathcal{D}^{1,2}(\R^n),
\end{equation}
where $n\geq 3$, $p:=\frac{n+2}{n-2}$ and $\mathcal{D}^{1,2}(\R^n)$ is the completion of $C_0^\infty(\R^n)$ with the norm $\|\nabla u\|_{L^2(\R^n)}$. Problem \eqref{prob} corresponds to  the steady state of the energy-critical focusing nonlinear wave equation
\begin{equation*}
\label{2m}
\partial_t^2 u-\Delta u- |u|^{\frac{4}{n-2}} u=0, \ (t,x)\in \R \times \R^n,
\end{equation*}
whose study (see for instance \cite{DJKM, DKM2, DKM3,  KM2, KST}) naturally relies on the complete classification of the set of non-zero finite energy solutions to  \eqref{prob}, which is defined by
\begin{equation*}
\label{3m}
\Sigma:= \left\{ Q \in {\mathcal D}^{1,2} (\R^n) \backslash \{0\}: \ -\Delta Q= \frac{n(n-2)}{4} |Q|^{\frac{4}{n-2}} Q\right \},
\end{equation*}
in particular
in connection with the {\it soliton resolution} conjecture for which only a few examples have become known \cite{KM1,KM2,DKM,DKM2,DKM3}. Observe that \eqref{prob} is the Euler-Lagrange equation of the functional defined by
\begin{equation}\label{energy}
{\bf e} (u) := {1\over 2} \int_{\R^n} |\nabla u |^2 \, dx - {\gamma (n-2) \over 2 n} \int_{\R^n} |u|^{2n \over n-2} \, dx.
\end{equation}
 Positive solutions to \eqref{prob} solve the Yamabe problem on the sphere (after a stereographic projection) and are the  extremal  functions for the  Sobolev embedding. Thanks to the classical work of Caffarelli-Gidas-Spruck \cite{CGS}, it is known that all positive solutions to (\ref{prob}) are given by the so called {\it bubble} and all its possible translations and dilations, that is,
\begin{equation}\label{bubble}
U(y):=\left(\frac{2}{1+|y|^2}\right)^{\frac{n-2}{2}} \mbox{ and }\;\;
U_{\alpha,y_0}(y):=\alpha^{-\frac{n-2}{2}}U\left(\frac{y-y_0}{\alpha}\right),\;\; \alpha>0,\; y_0\in\mathbb{R}^n,
\end{equation}
previously discovered independently by Aubin \cite{Au} and Talenti \cite{Ta}. In fact, all radial solutions in $\Sigma$ have the form \eqref{bubble}.
Sign-changing solutions belonging to $\Sigma$ have thus to be non radial. Using the classical theory of Ljusternik-Schnirelman  category, Ding proved in \cite{D} the existence of infinitely many elements in  $\Sigma$  that are non radial, sign-changing, and with arbitrary large energy.
The key idea in \cite{D} is to look for solutions to \eqref{prob} that
are invariant under the action of $O(2) \times O(n-2) \subset O(n)$ to recover  compactness for the functional ${\bf e} (u)$.
No further information though is known on the solutions found by Ding. Recently, more explicit constructions for  sign-changing (non radial) solutions to \eqref{prob} have been obtained by del Pino-Musso-Pacard-Pistoia and Medina-Musso-Wei (see \cite{dPMPP, dPMPP2, MMW}).
The solutions obtained in \cite{dPMPP} are invariant under the action of $D_k \times O(n-2)$, where $D_k$ is the dihedral group of rotations and reflections leaving a regular polygon with $k$ sides invariant.
More precisely, for any $k$ large enough, the authors construct a solution to \eqref{prob}  looking like the {\it bubble} $U$ in \eqref{bubble} surrounded by $k$ negative scaled copies of $U$ arranged along the vertices of a $k$-regular polygon in $\R^2$. At main order the solution looks like
\begin{equation}\label{circle}
U(y)-\sum_{j=1}^k\lambda^{-\frac{n-2}{2}}U\left(\lambda^{-1}(y-\xi_j)\right),
\end{equation}
where $\xi_j:=(e^{\frac{2\pi (j-1)i}{k}},0,\ldots)$, $\lambda=O(k^{-2})$ if $n\geq 4$ and $\lambda=O((k\ln k)^{-2})$ if $n=3$, as $k \to \infty$.
Observe that
$$
\sum_{j=1}^k \lambda^{-\frac{n-2}{2}}U\left(\lambda^{-1}(y-\xi_j)\right) \rightharpoonup c_n \delta_\Gamma, \quad {\mbox {as}} \quad k \to \infty,
$$
for a positive constant $c_n$, where $\delta_\Gamma$ is the Dirac-delta at the equatorial on the $(y_1 , y_2)$-plane $\Gamma =\{ y \in \R^n \, : \, y_1^2 + y_2^2=1 \}$ in  $S^n(1)$. We thus can think of the solutions obtained in \cite{dPMPP} and  described at main order in \eqref{circle} as the sum of a positive fixed central bubble surrounded by a negative smooth function that desingularizes a Dirac-delta along the equatorial $\Gamma$, in the limit as $k \to \infty$. We call this construction a {\it desingularization of the equatorial}, in analogy with similar desingularization constructions for
minimal surfaces in Riemannian three-manifolds \cite{K}.
We remark that these solutions are not
invariant under the action of $O(2) \times O(n-2)$, thus they differ from the ones found by Ding.
Besides, these solutions are the first example of non degenerate sign-changing  solutions to \eqref{prob} as shown in \cite{MW}.
For any $Q \in \Sigma$, consider the linear operator $L_Q:=-\Delta -\gamma p |Q|^{p-1}$ and define the null space
$$
{\mathcal Z}_Q := \{ f \in {\mathcal D}^{1,2} (\R^n) \, : \, f \not=0, \, L_Q (f) = 0\}.
$$
Duyckaerts-Kenig-Merle \cite{DKM} introduced the following definition of non-degeneracy for a solution of problem \eqref{prob}: $Q\in \Sigma$ is said to be {\em non degenerate} if ${\mathcal Z}_Q$ coincides with the vector space generated by the elements in  ${\mathcal Z}_Q$ related to the group of isometries in
${\mathcal D}^{1,2} (\R^n)$ under which problem \eqref{prob} is invariant, given by translations, scalings, rotations and Kelvin transformation.
More precisely, $Q$ is non degenerate if
$$
{\mathcal Z}_Q = \tilde{{\mathcal Z}}_{Q},
$$
where
$$
\tilde{{\mathcal Z}}_{Q}:= \mbox{span} \ \left\{ \begin{array}{l}
(2-n) x_j Q +|x|^2 \partial_{x_j} Q- 2 x_j x \cdot \nabla Q, \quad  \partial_{x_j} Q, \quad 1\leq j\leq n, \\
\\
(x_j \partial_{x_k} - x_k \partial_{x_j})Q, \quad 1\leq j <k\leq n, \quad \frac{n-2}{2} Q + x \cdot Q \end{array}
\right\}.
$$
The rank of a solution $Q \in \Sigma$ is the dimension of the vector space $\tilde {\mathcal Z}_Q$, and it cannot exceed the number
$$\mathcal{N}:=2n+1+\frac{n(n-1)}{2},$$
being this the largest possible dimension for $\tilde {\mathcal Z}_Q$.
It is well known for instance that
$U$ in \eqref{bubble} is non degenerate and its rank is $n+1$ (see \cite{rey}).
In \cite{MW} it is proven that the solutions built in \cite{dPMPP}, looking  at main order as in  \eqref{circle}, are non degenerate and their rank is $3n$.

A question we address in \cite{MMW} is about the existence of solutions $Q \in \Sigma$ to \eqref{prob} whose rank is maximal.  In fact, observe that not the bubbles in \eqref{bubble} nor the solutions built in \cite{dPMPP} have maximal rank. We partially answer this question building a  new family of solutions to \eqref{prob}, where another polygon with a large number of sides is replicated for $n\geq 4$ in the third and fourth coordinates, giving rise to a new family of non degenerate solutions to \eqref{prob} that at main order look like
\begin{equation}\label{circle2}
U(y)-\sum_{j=1}^k\lambda^{-\frac{n-2}{2}}U(\lambda^{-1}(y-\oxi_j))-\sum_{j=1}^h\mu^{-\frac{n-2}{2}}U(\mu^{-1}(y-\uxi_j)),
\end{equation}
where $\oxi_j:=(e^{\frac{2\pi (j-1)i}{k}},0,\ldots)$, $\uxi_j:=(0,0, e^{\frac{2\pi (j-1)i}{h}},0,\ldots)$, $\lambda=O(k^{-2})$ and $\mu=O(h^{-2})$, for $k$ and $h$ sufficiently large, see \cite{MMW}. Using the terminology we introduced before,
this construction is a {\it desingularization of the two equators}, one in the $(y_1, y_2)$-plane and the other in the $(y_3 , y_4)$-plane. These solutions are non degenerate. Furthermore,
their rank is $5 \, (n-1)$, which is {\it maximal} when $n=4$.  A generalization of this result is to consider  non degenerate solutions obtained by gluing a central positive bubble with  negative  scaled copies of  $U$ centered at the vertices of $\ell$ regular polygons with a
large number of sides and lying in consecutive planes, which have maximal rank  provided the dimension $n$ is  $2\ell$. In other words, a {\it desingularization of $\ell$ equators} in consecutive planes would provide an example of non degenerate solutions with maximal rank in any even dimension $2\ell$. If the dimension $n$ is odd, the existence of non degenerate solutions for \eqref{prob} with maximal rank  remains an open problem and a  different construction is required.

Roughly speaking, the solution built in \cite{MMW} {\it breaks} the radial behavior of the bubble in the first four coordinates, loosing the related invariances. This fact adds extra terms in the kernel of the linearized operator, being {\it all the possible} precisely when $n=4$. Thus, to prove the analogue in odd dimension one needs to find a solution {\it breaking the radiality} in an odd number of coordinates.

The aim of this work is to address this question and to provide a new family of sign-changing solutions for problem \eqref{prob} that we claim to have maximal rank  in dimension $n=3$. We prove the following result.

\begin{theorem}\label{theo}
Let $n\geq 3$ and let $k$ be a positive integer. Then for any sufficiently large $k$ there is a finite energy solution to \eqref{prob} of the form
$$u(y)=U(y)-\sum_{j=1}^k\lambda^{-\frac{n-2}{2}}U(\lambda^{-1}(y-\oxi_j))-\sum_{j=1}^k\lambda^{-\frac{n-2}{2}}U(\lambda^{-1}(y-\uxi_j))+o_k(1)(1+\lambda^{-\frac{n-2}{2}}),$$
where
$$\oxi_j:=R(\sqrt{1-\tau^2}e^{\frac{2\pi(j-1)}{k}i},\tau,0,\ldots),\; \;\uxi_j:=R(\sqrt{1-\tau^2}e^{\frac{2\pi(j-1)}{k}i},-\tau,0,\ldots), \;\;j=1,\ldots,k,$$
and
\begin{equation*}\begin{split}
\la := {\ell^{2\over n-2} \over k^2} ,  \; \tau := {t \over k^{1-{2\over n-1}} },  \;{\mbox {if}} \quad n\geq 4, \qquad
\la := {\ell^{2} \over k^2 (\ln k)^2 } , \; \tau := {t \over \sqrt{\ln k} },  \; {\mbox {if}} \quad n=3.
\end{split}
\end{equation*}
Here
$$\lambda^2+R^2=1, \quad {\mbox {and}} \quad  \eta <  \ell , \, t < \eta^{-1},
 $$
 for some positive fixed number $\eta$ independent of $k$.  The term $o_k(1)\rightarrow 0$ uniformly on compact sets of $\R^n$ as  $k \to \infty$.

 These solutions  have maximal rank in dimension $3$.
\end{theorem}
Some remarks are in order.

\begin{remark}
Let us briefly discuss our construction.
The solution predicted by Theorem \ref{theo} looks at main order as
$$
u_0(y) := U(y)-\sum_{j=1}^k\lambda^{-\frac{n-2}{2}}U(\lambda^{-1}(y-\overline \xi_j))-\sum_{j=1}^k\lambda^{-\frac{n-2}{2}}U(\lambda^{-1}(y-\underline \xi_j)).
$$ 
The polygonal distribution of the points  $\overline \xi_1, \ldots , \overline \xi_k$ and 
$\underline \xi_1, \ldots \underline \xi_k$ makes $u_0$ a function with several important symmetries: it is invariant under rotation of angle ${2\pi \over k}$ in the $(y_1, y_2)$-plane, and even in the other variables $y_j$, $j=3, \ldots , n$. The assumption that $\lambda^2 + R^2=1$ gives that $u_0$ is also invariant under Kelvin transformation. We will take great advantage of these symmetries in many different ways in our proof. For instance they allow us to choose the same scaling factor $\lambda$ for each one of the negative bubbles centred at the different points  $\overline \xi_j$ and $\underline \xi_j$, $j=1, \ldots , k$, which in principle may not be the same, reducing substantially the number of scaling parameters to adjust.
Taking $\lambda$ and $\tau$ small positive parameters as $k \to \infty$, a formal computation shows that the energy functional defined in \eqref{energy} and evaluated at $u=u_0$ has the following expansion, as $k \to \infty$,
\begin{align*}
{\bf e} (u_0)& \sim (2k +1) a_n + k \left(  \lambda^{n-2 \over 2} (b_n- c_n \tau^2 ) - d_n \lambda^{n-2} k^{n-2} - e_n { \la^{n-2} \over \tau^{n-3}} \right) ,
\end{align*}
when dimension $n\geq 4$, for some explicit positive constants $a_n$, $b_n$, $c_n$, $d_n$ and $e_n$. Our choice for $\lambda$ and $\tau$ in terms of $k$ is to get at main order the balance
$$
\nabla_{\lambda , \tau} {\bf e} (u_0) \sim 0.
$$
We will justify this heuristic argument in Section 2.
\end{remark}

\begin{remark}
The construction obtained in Theorem \ref{theo} differs from the {\it desingularization of the equatorial} in \eqref{circle} obtained in \cite{dPMPP} or of two equators in \eqref{circle2} in \cite{MMW}.
In fact the solutions in Theorem \ref{theo} can be thought as the sum of a positive fixed central bubble surrounded by a negative smooth function that desingularizes  Dirac-deltas located in points on two circles that are collapsing into a Dirac-delta supported along the equatorial $\Gamma$, in the limit as $k \to \infty$. We call this construction a {\it doubling of the equatorial} in the $(y_1 , y_2)$-plane, in analogy with similar doubling constructions for
minimal surfaces in Riemannian three-manifolds obtained in \cite{K,K1,K2}.
\end{remark}
\begin{remark}\label{li} Let $u$ be the solution predicted in Theorem \ref{theo} and define
the following $4n-2$ functions
$$z_0(y):=\frac{n-2}{2}u(y)+\nabla u(y)\cdot y,\quad z_\alpha(y):=\frac{\partial}{\partial y_\alpha}u(y),\;\alpha=1,\ldots,n,$$
$$z_{n+1}(y):=-2y_1z_0(y)+|y|^2z_1(y),\quad z_{n+2}(y):=-2y_2z_0(y)+|y|^2z_2(y),$$
$$ z_{n+3}(y):=-2y_3z_0(y)+|y|^2z_3(y),$$
$$z_{n+\alpha+2}(y):=-y_\alpha z_1(y)+y_1z_\alpha(y),\;\alpha=2,\ldots, n,$$
$$z_{2n+\alpha}(y):=-y_\alpha z_2(y)+y_2z_\alpha(y),\;\alpha=3,\ldots, n,$$
$$z_{3n+\alpha-3}(y):=-y_\alpha z_3(y)+y_3z_\alpha(y),\;\alpha=4,\ldots, n.$$
These functions belong to $\mathcal{D}^{1,2}(\R^n)$ and also
 to the kernel of $L_u:=-\Delta-\gamma p |u|^{p-1}$. Furthermore, if $k$ is even they are linearly independent (see Appendix B). Consider now the case $n=3$: we have that $4n-2=\mathcal{N}=10$. Thus the solutions constructed in Theorem \ref{theo} have maximal rank in dimension $3$ when $k$ is even. The non-degeneracy of the solution remains an open problem.
\end{remark}
\begin{remark}
The {\it doubling of the equatorial} in $\R^3$ presumably closes the question about solutions with maximal rank. Indeed, for any odd dimension we could combine a {\it doubling} of the equatorial in three coordinates with a {\it desingularization} of the equatorial as in \cite{dPMPP}  in two coordinates or a {\it desingularization} of two equators as in \cite{MMW} in four coordinates, as many times as needed. We conjecture that combining these three structures we can build a maximal solution in any odd dimension.
\end{remark}
\begin{remark}
This {\it doubling} construction also provides an alternative example of solution with maximal rank for even dimensions of the form $n=6\ell$, $\ell \in\mathbb{N}$. For these dimensions, in \cite{MMW} the authors propose to replicate $3\ell$ times the {\it desingularization} of \cite{dPMPP}. We conjecture that combining $2\ell$ structures like the {\it doubling of the equatorial} of Theorem \ref{theo} would provide a different solution with maximal rank in these dimensions.
\end{remark}
\begin{remark}\label{r16}
The construction described in Theorem \ref{theo} is not the only possible way to {\it double the equatorial} $\Gamma$.
For any even integer $2m$, we can construct another sequence of solutions that {\it doubles the equatorial} in the form of  the sum of a positive fixed central bubble surrounded by a negative smooth function that desingularizes  Dirac-deltas located in points on $2m$ circles that are collapsing into a Dirac-delta supported along the equatorial $\Gamma$, in the limit as $k \to \infty$, where $m$ of these circles collapse onto $\Gamma$
from above and $m$ from below.

We can also combine {\it desingularization of the equatorial} and {\it doubling of the equatorial}.
For any odd integer  $2m +1$, we can construct a sequence of solutions with the form of the sum of a positive fixed central bubble surrounded by a negative smooth function that consists of two parts. One part desingularizes  Dirac-deltas located in points on $2m$ circles that are collapsing into a Dirac-delta supported along the equatorial $\Gamma$, in the limit as $k \to \infty$. The other part desingularizes Dirac-deltas located at points along the equatorial,
desingularizing a Dirac-delta along the equatorial $\Gamma$, in the limit as $k \to \infty$.

Since the proofs of these constructions are in the same spirit as the one of Theorem \ref{theo}, we will briefly describe them in section \ref{general}, explaining the principal differences.
\end{remark}

The proof of Theorem \ref{theo} is based on a Lyapunov-Schmidt reduction in the spirit of \cite{dPMPP}: we define a first approximation and we look for a solution in a nearby neighborhood, by  linearizing around the approximation and, after developing an appropriate linear theory, solving by a fixed point argument. This allows to reduce the original problem to the solvability of a finite dimensional one. However, for the construction in Theorem \ref{theo} this last step is rather delicate. In fact, the finite dimensional reduction leads to two equations (in the  two parameters to adjust, $\lambda$ and $\tau$) where the sizes of the error and the non linear term play a fundamental role. If one follows the strategy of \cite{dPMPP}, the nonlinear term cannot be controlled and the reduced problem cannot be solved.  For this reason, we need to carry on a much more refined argument.

The key point is the following: if one pays attention to the error term near the bubbles, this can be decomposed in a (relatively) large but symmetric part, and a smaller but non symmetric part (see Proposition \ref{symErrInt}). In the final  argument (the reduction procedure) the symmetric part is orthogonal to the element of the kernel, and therefore not seen. Thus, the part playing a role in the reduction is the non symmetric one, that is significatively smaller. Roughly speaking, this allows to solve the linearized problem in two parts, a symmetric and large but irrelevant part, and a non symmetric but small one. Indeed, if our solution has the form $u_*+\phi$, being $u_*$  the approximation, we will split $\phi$ near each bubble as $\phi=\phi^s+\phi^*$, where $\phi^s$ is symmetric with respect to the hyperplane $y_3=\tau$ (or analogously $y_3=-\tau$). This behavior is also inherited by the nonlinear part of the equation and we will be able to perform the fixed point argument setting the size of $\phi^*$ very small (see Proposition \ref{existPhi1}), what will allow us to conclude the reduction. This new strategy requires delicate decompositions and estimates of every term of the equations, as well as the development of a sharp linear invertibility theory.

\medskip
The structure of the article is the following. In section \ref{secError} we detail the approximation and the error associated, estimating it near and far from the bubbles in different norms, and identifying the symmetric and non symmetric parts and their sizes. Section \ref{linear} is devoted to the linear theory, where a refinement of the theory in \cite{dPMPP} is developed. Section \ref{gluing} is the core of the strategy, where the gluing scheme is performed, together with the precise decomposition of the function in their symmetric and non symmetric part. In section \ref{thm} we carry on the dimensional reduction, concluding the proof of Theorem \ref{theo}. The appendix contains fundamental computations concerning the shape of the approximation.

\section{Doubling construction: a first approximation}\label{secError}

\noindent Let $n \geq 3$ and $\tau \in (0,1)$. In $\R^n$ we fix the following points
$$
 \overline P := (\sqrt{1-\tau^2 } , 0 , \tau , 0 , \ldots , 0) , \quad \underline P :=  (\sqrt{1-\tau^2 } , 0 , -\tau , 0 , \ldots , 0).
$$
Let $ \la \in (0,1)$ be a positive number, and define $R$ as
\begin{equation}\label{defR}
\la^2 + R^2 = 1.
\end{equation}
Let $k$ be an integer number and
\begin{equation}\label{ansatz}
\textbf{u} [\la, \tau ] (y) := U(y) -\sum_{j=1}^k \underbrace{\la^{-{n-2 \over 2}} U\left( {y-\oxi_j \over \la} \right)}_{\oU_j (y)}
-\sum_{j=1}^k \underbrace{\la^{-{n-2 \over 2}} U\left( {y- \uxi_j \over \la} \right)}_{ \uU_j (y) }
\end{equation}
where
\begin{equation}\begin{aligned}\label{defOxi}
\oxi_j &:= R (\sqrt{1-\tau^2} \cos \theta_j , \sqrt{1-\tau^2} \sin \theta_j , \tau , 0, \ldots , 0 ), \\
 \uxi_j &:= R (\sqrt{1-\tau^2} \cos \theta_j , \sqrt{1-\tau^2} \sin \theta_j , -\tau , 0, \ldots , 0 ),
\end{aligned} \quad {\mbox {with}} \quad \theta_j := 2\pi {j-1 \over k}.
\end{equation}
Observe that $\oxi_1 = R\overline P$ and $\uxi_1 = R\underline P$, while $\oxi_j$ and $ \uxi_j$ are obtained respectively from $\overline P$ and $\underline P$ after a rotation in the $(y_1 , y_2)$-plane of angle $2\pi \, {{j-1} \over k} $. Thanks to \eqref{defR}, the functions $\oU_j$, $\uU_j$ are invariant under Kelvin transform,
so that $\textbf{u}$ is also invariant under this transformation, that is
\begin{equation}\label{ikt}
\textbf{u} (y) = |y|^{2-n} \textbf{u} \left({y\over |y|^2} \right).
\end{equation}
A direct observation reflects that $\textbf{u}$ also shares the following symmetries
\begin{equation}\label{ip}
\textbf{u} (y_1, \ldots , - y_j , \ldots  , y_n) = \textbf{u} (y_1, \ldots , y_j , \ldots  , y_n), \quad j=2, \ldots , n,
\end{equation}
\begin{equation}\label{ir}
\textbf{u} (e^{2\pi { (j-1)\over k} i} \bar y , y_3, \ldots , y_n) = \textbf{u} ( \bar y , y_3, \ldots , y_n), \quad \bar y= (y_1 , y_2), \quad j=2, \ldots , k.
\end{equation}
In our construction, we assume that the integer $k$ is large and the parameters  $\la $ and $\tau$ are given by
\begin{equation}\begin{split} \label{par1}
\la &:= {\ell^{2\over n-2} \over k^2} ,  \quad  \tau := {t \over k^{1-{2\over n-1}} },  \quad {\mbox {if}} \quad n\geq 4, \\
\la &:= {\ell^{2} \over k^2 (\ln k)^2 } , \quad  \tau := {t \over \sqrt{\ln k} },  \quad {\mbox {if}} \quad n=3,
\end{split}  \quad
\quad {\mbox {
where}} \quad
\eta <  \ell , \, t < \eta^{-1},
\end{equation}
for some $\eta$ small and fixed, independent of $k$, for any $k$ large enough. The error function, defined as
\begin{equation}\label{error}
E [\ell , t] (y) := \Delta \textbf{u} + \gamma |\textbf{u}|^{p-1} \textbf{u}, \quad y \in \R^n,
\end{equation}
inherits the symmetries \eqref{ip}, \eqref{ir}. As a consequence, fixing a small $\delta >0$ independent of $k$, it is enough to describe the error function
in the sets $B(\oxi_1 , {\delta \over k})$
and $\R^n \setminus \bigcup_{j=1}^k \left( B(\oxi_j , {\delta \over k})  \cup B( \uxi_j , {\delta \over k}) \right)  $
in order to know it in the whole space $\R^n$. 

For our purpose it is convenient to measure the error using the following weighted $L^q$ norm
\begin{equation}\label{normstar2}
\|h\|_{**} :=  \|  \, (1+ |y|)^{{n+2} - \frac {2n} q} h \|_{L^q(\R^n)}.
\end{equation}
For the moment we request that $q$ is a fixed number with  $  {n \over 2} <q<n$. Later on, we will need a more restrictive assumption on $q$, $\frac n2<q<\frac{n}{2-\frac{2}{n-1}}$. We will evaluate the
 $\| \cdot \|_{**}$-norm of  $E$ in the {\it interior regions} $
  B(\oxi_j, {\delta \over k} )$ and $ B( \uxi_j, {\delta \over k} )$, for any $j=1, \ldots , k$, and
in the {\it exterior region} $\R^n \setminus \bigcup_{j=1}^k \left( B(\oxi_j, {\delta \over k}) \cup B( \uxi_j, {\delta \over k} ) \right)$. 

{\it The error in the interior regions $B (\oxi_j , {\delta \over k} )$ and $ B( \uxi_j, {\delta \over k} )$, $j=1, \ldots , k$}. To describe the error $E$ in each one of the balls
$B(\oxi_j, {\delta \over k})$, $ B( \uxi_j, {\delta \over k} )$, it is enough to do it in $B(\oxi_1, {\delta \over k})$, as already observed. In this region the dominant term of the function $\textbf{u}$ in \eqref{ansatz} is  $\oU_1$. Thus, for some $s \in (0,1)$, we have
\begin{equation*}\begin{split}
\gamma^{-1} E (x) &= p \left( \oU_1 (x) + s [ \sum_{j\not=1} \oU_j (x) + \sum_{j=1}^k \uU_j (x) - U(x) ] \right)^{p-1} \left[ -\sum_{j\not=1} \oU_j (x) - \sum_{j=1}^k \uU_j (x) + U(x) \right] \\
&+ \sum_{j\not= 1} \oU_j^p (x) + \sum_{j=1}^k \uU_j^p (x) - U^p (x),
\end{split}
\end{equation*}
for $x \in B (\oxi_1 , {\delta \over k} )$. Let us introduce the change of variable $\la y = x-\oxi_1$, so that
$\oU_1 (\la y + \oxi_1 )= \la^{-{n-2 \over 2}} U(y)$.
In these expanded variables, the error takes the form
\begin{equation}\begin{split} \label{e1}
\la^{n+2 \over 2} \gamma^{-1} E (\oxi_1 + \la y ) &= p \left( U (y) + s \la^{n-2 \over 2} [ \sum_{j\not=1} \oU_j (\la y + \oxi_1 ) + \sum_{j=1}^k \uU_j (\la y + \oxi_1) - U(\la y + \oxi_1 ) ] \right)^{p-1} \times \\
&\times \la^{n-2 \over 2} \, \left[ \sum_{j\not=1} \oU_j (\la y + \oxi_1 ) + \sum_{j=1}^k \uU_j (\la y + \oxi_1 ) - U(\la y + \oxi_1 ) \right] \\
&+ \la^{n+2 \over 2} \left[ \sum_{j\not= 1} \oU_j^p (\la y + \oxi_1 ) + \sum_{j=1}^k \uU_j^p (\la y + \oxi_1 ) - U^p (\la y + \oxi_1 )\right],
\end{split}
\end{equation}
for some $s \in (0,1)$, and
uniformly for $|y|<{\delta \over \la k}$. A direct Taylor expansion gives that

\begin{equation}\begin{split} \label{e0}
\oU_j (\la y + \oxi_1) &= {2^{n-2 \over 2} \la^{n-2 \over 2} \over (\la^2 +|\la y + \oxi_1 - \oxi_j |^2)^{n-2 \over 2} }\quad {\mbox {for}} \quad j\not=1  \\
&= {2^{n-2 \over 2} \la^{n-2 \over 2} \over |\oxi_1- \oxi_j|^{n-2}} \left[ 1- (n-2) {(y, \oxi_1-\oxi_j) \over |\oxi_1 - \oxi_j|^2 } \la +
O\left( {\la^2 (1+ |y|^2) \over |\oxi_1 - \oxi_j |^2} \right) \right],\\
\uU_j (\la y + \oxi_1) &= {2^{n-2 \over 2} \la^{n-2 \over 2} \over (\la^2 +|\la y + \oxi_1 - \uxi_j |^2)^{n-2 \over 2} }\quad {\mbox {for}} \quad j=1 , \ldots , k  \\
&= {2^{n-2 \over 2} \la^{n-2 \over 2} \over |\oxi_1-  \uxi_j|^{n-2}} \left[ 1- (n-2) {(y, \oxi_1- \uxi_j) \over |\oxi_1 -  \uxi_j|^2 } \la +
O\left( {\la^2 (1+ |y|^2) \over |\oxi_1 -  \uxi_j |^2} \right) \right],\\
U(\la y + \oxi_1)&= U(\oxi_1 ) \left[ 1- (n-2) {(y, \oxi_1 ) \over 1+ |\oxi_1 |^2} \la + O\left({\la^2 |y|^2 \over 1+ |\oxi_1|^2} \right) \right],
\end{split}
\end{equation}
uniformly for $|y| \leq {\delta \over \la k}$. In the Appendix we will show that
\begin{equation}\begin{split} \label{uno}
\sum_{j=2}^k {1\over |\oxi_1 - \oxi_j|^{n-2}} =\sum_{j=2}^k {1\over |\uxi_1 - \uxi_j|^{n-2}} =\begin{cases} A_n \,  k^{n-2} \,  \left(1+  O( \tau^2) \right) \quad
{\mbox {if}} \quad n \geq 4,\\ A_3 \, k  \, \ln k  \,   \left(1+ O(\tau^2 ) \right) \quad {\mbox {if}} \quad n =3,
\end{cases}
\end{split}
\end{equation}
and
\begin{equation}\begin{split} \label{due}
\sum_{j=1}^k {1\over |\oxi_1 - \uxi_j|^{n-2}} =
\begin{cases}  B_n \, {k \over \tau^{n-3}} \, \left(1+  O( (\tau k)^{-2}) \right) \quad
{\mbox {if}} \quad n \geq 5,\\
B_n \, {k \over \tau^{n-3}} \, \left(1+  O( \tau) \right) \quad
{\mbox {if}} \quad n =4,\\
A_3 \,  k \,  \ln \left(\frac{\pi}{\tau}\right)\,   \left(1+ O(|\ln \tau|^{-1} ) \right) \quad {\mbox {if}} \quad n =3,
\end{cases}
\end{split}
\end{equation}
where $A_3 := \pi^{-1}$ and, if $n\geq 4$,
$$
A_n:=  {2\over (2\pi)^{n-2}} \, \sum_{j=1}^\infty j^{2-n},\; \quad B_n := {2\over 2^{n-2} \pi}  \int_0^\infty {ds \over (1+ s^2)^{n-2 \over 2}}.
$$
From \eqref{e1}, \eqref{e0}, \eqref{uno} and \eqref{due} we get that
$$
| \la^{n+2 \over 2} \gamma^{-1} E (\oxi_1 + \la y ) |\le  C \left [  \frac { \la^{\frac{n-2}2}}{1+ |y|^4}   +  \la^{\frac{n+2}2}  \right ]
$$
for $|y| \leq {\delta \over \la k}$. Notice that, to obtain this estimate, we do not need to use the precise information gathered in \eqref{uno} and \eqref{due}, but only the order in $k$.  Indeed, it is enough to get the upper bound
$$|\oU_j (\la y + \oxi_1)|\leq C\quad \mbox{ for }|y| \leq {\delta \over \la k}, $$
and the estimate follows since, as \eqref{due} reflects, $|\uU_j (\la y + \oxi_1)|$ is smaller.

Direct computations (see \cite{dPMPP}) give
$$
\|\, (1+ |y|)^{{n+2} - \frac {2n} q}\, \la^{n+2 \over 2} \gamma^{-1} E (\oxi_1 + \la y ) \|_{L^q( |y| < {\delta \over \la k})} \ \le\  \begin{cases}
C
k^{-{n\over q}} \quad {\mbox {if}} \quad n\geq 4,\\
 {C \over k \ln k}
\quad {\mbox {if}} \quad n=3,
\end{cases}
$$
for some fixed constant $C>0$. Therefore, by symmetry we conclude that, for any $j=1, \ldots , k$,
\begin{equation}
\label{formufinal}
\|\, (1+ |y|)^{{n+2} - \frac {2n} q}\, \la^{n+2 \over 2} \gamma^{-1} E (\oxi_j + \la y ) \|_{L^q( |y| < {\delta \over \la k})} \ \le \
\left\{\begin{array}{cc}
Ck^{- \frac{n} {q} }  \quad {\mbox {if}} \quad n\geq 4,\\
{C\over k \ln k} \quad {\mbox {if}} \quad n=3,
 \end{array}
 \right.
\end{equation}
and
\begin{equation}
\label{formufinal2}
\|\, (1+ |y|)^{{n+2} - \frac {2n} q}\, \la^{n+2 \over 2} \gamma^{-1} E (\uxi_j + \la y ) \|_{L^q( |y| < {\delta \over \la k})} \ \le \
\left\{\begin{array}{cc}
Ck^{- \frac{n} {q} }  \quad {\mbox {if}} \quad n\geq 4,\\
{C\over k \ln k} \quad {\mbox {if}} \quad n=3.
 \end{array}
 \right.
\end{equation}

\medskip
\noindent {\it The error in $\R^n \setminus \bigcup_{j=1}^k \left( B(\uxi_j , {\delta \over k}) \cup B(\bar \xi_j , {\delta \over k}) \right)$}. For
$y$ in this region
we have
\begin{equation}\begin{split}\label{errorFuera}
|E (y)| &\le   C \left [  \frac 1 {(1+ |y|^2)^2}  +  \left| \sum_{j=1}^k  \left[ \frac {\la^{\frac{n-2}2 } }{ |y-\uxi_j|^{n-2}} + \frac {\la^{\frac{n-2}2 } }{ |y-\bar \xi_j|^{n-2}}\right]  \right|^{\frac 4{n-2}} \right ] \, \\
&\times \left [ \sum_{j=1}^k  \left( \frac { {\la^{\frac{n-2}2 }}  } { |y-\uxi_j|^{n-2}} +
\frac {\la^{\frac{n-2}2 } }{ |y-\oxi_j|^{n-2}}\right) \right]\\
&
\le \  C \frac {\la^{\frac{n-2}2 }} {(1+ |y|^2)^2} \sum_{j=1}^k  \left( \frac {1  } { |y-\uxi_j|^{n-2}}   +
\frac {1 }{ |y-\oxi_j|^{n-2}}\right),
\end{split}\end{equation}
where we have used that in this exterior region
\begin{equation}\label{sumFuera}\sum_{j=1}^k\frac { {\la^{\frac{n-2}2 }}  } { |y-\uxi_j|^{n-2}}\leq C,\quad \sum_{j=1}^k\frac {\la^{\frac{n-2}2 } }{ |y-\oxi_j|^{n-2}}\leq C.
\end{equation}
For $n\ge 4$, using \eqref{par1},
\begin{equation}\begin{split}\label{e5}
\| \, &(1+ |y|)^{{n+2} - \frac {2n} q}\, E\|_{L^q( \R^n \setminus \bigcup_{j=1}^k \left( B(\uxi_j , {\delta \over k}) \cup B(\bar \xi_j , {\delta \over k}) \right))} \ \\
&\le\
C\la^{\frac{n-2}2} \sum_{j=1}^k \left[
\left (    \int_ { |y-\uxi_j| > {\delta\over k}}  \frac {(1+ |y|)^{(n+2)q - 2n-4q}} {|y-\uxi_j|^{(n-2)q }}\, dy\,   \right )^{\frac 1q }  \right.\\
&\;\;\;\;\left.+ \left( \int_ { |y-\bar \xi_j| > {\delta\over k}}  \frac {(1+ |y|)^{(n+2)q - 2n-4q}}  {|y-\oxi_j|^{(n-2)q }}\, dy\,   \right )^{\frac 1q } \right] \\
&\le \  C
\la^{\frac{n-2}2}  k
\left (    \int_{  {\delta \over k} }^1
\frac { t^{n-1} } {t^{(n-2)q} }\, dt\,     \right )^{\frac 1q }\
 \le \, Ck^{1-\frac nq},
\end{split} \end{equation}
for some constant $C$.
Similarly, if $n=3$, we get
\begin{equation}
\| (1+ |y|)^{{n+2} - \frac {2n} q}\, E \|_{L^q( \R^n \setminus \bigcup_{j=1}^k ( B(\uxi_j , {\delta \over k}) \cup B(\bar \xi_j , {\delta \over k}) ))} \ \le\  \frac {C} {\ln k  }\ .
\label{ee5}\end{equation}
Estimates \eqref{formufinal}, \eqref{formufinal2}, \eqref{e5} and \eqref{ee5} merge in the following

\begin{prop}\label{Pest2}
Assume that $\la$ and $\tau$ satisfy  \eqref{par1}.
There exist an integer $k_0$ and a positive constant $C$ such that for all $k\geq k_0$ the following estimates
hold true
$$
\| E \|_{**} \leq C k^{1-{n\over q}} \quad {\mbox {if}} \quad n\geq 4
\quad
{\mbox {and}}
\quad
\| E \|_{**} \leq C |\ln k|^{-1} \quad {\mbox {if}} \quad n = 3.
$$
We refer to \eqref{normstar2} for the definition of the $\| \cdot \|_{**}$-norm.
\end{prop}

The following result is the key point of the argument carried out for the gluing procedure and the reduction method in sections \ref{gluing} and \ref{thm}. It provides a decomposition of the interior error term in two parts: one is relatively large but symmetric, and the other is non symmetric but smaller in size. Thanks to this observation we can refine the fixed point argument to fix a smaller size of the functions involved in the reduction.
\begin{prop}\label{symErrInt}
Assume that $\la$ and $\tau$ satisfy  \eqref{par1} and $|y|<\frac{\delta}{\lambda k}$. Then, there exists a decomposition
$$E(\oxi_1+\lambda y)=E^s(\oxi_1+\lambda y)+E^*(\oxi_1+\lambda y),$$
such that $E^s(\oxi_1+\lambda y)$ is even with respect to $y_\alpha$ for all $\alpha=1,\ldots,n$
and there exists an integer $k_0$ such that, for $k\geq k_0$,
\begin{equation}\label{errPoint}
|\lambda^{\frac{n+2}{2}}E^*(\oxi_1+\lambda y)|\leq 
\begin{cases}
C \frac{\lambda^{\frac{n-2}{2}}}{k}\frac{1}{1+|y|^3}\mbox{ if }n\geq 4,\\
C \frac{\lambda^{1/2}}{k(\ln k)^3}\frac{1}{1+|y|^3}\mbox{ if }n=3,
\end{cases},\qquad |y|<\frac{\delta}{\lambda k}.
\end{equation}
\end{prop}
\begin{proof}
The idea is to identify in \eqref{e1} the main order terms using \eqref{e0}. Indeed, we write, for $|y|<\frac{\delta}{k\lambda}$,
\begin{equation}\begin{split}\label{errorSim}
\gamma^{-1} \, \lambda^{\frac{n+2}{2}}&E^s(\oxi_1+\lambda y):=pU(y)^{p-1}\lambda^{\frac{n-2}{2}}\left[\sum_{j\neq 1}\frac{2^{\frac{n-2}{2}}\lambda^{\frac{n-2}{2}}}{|\oxi_1-\oxi_j|^{n-2}}+\sum_{j=1}^k\frac{2^{\frac{n-2}{2}}\lambda^{\frac{n-2}{2}}}{|\oxi_1-\uxi_j|^{n-2}}-U(\oxi_1)\right]\\
&+ \la^{n+2 \over 2} \left[ \sum_{j\not= 1} \left(\frac{2^{\frac{n-2}{2}}\lambda^{\frac{n-2}{2}}}{|\oxi_1-\oxi_j|^{n-2}}\right)^p+ \sum_{j=1}^k \left(\frac{2^{\frac{n-2}{2}}\lambda^{\frac{n-2}{2}}}{|\oxi_1-\uxi_j|^{n-2}}\right)^p - U^p (\oxi_1 )\right],
\end{split}\end{equation}
that is even with respect to every $y_\alpha$, $\alpha=1,\ldots,n$. Let us define
$$E^*(\oxi_1+\lambda y):=E(\oxi_1+\lambda y)-E^s(\oxi_1+\lambda y).$$
Using \eqref{uno} (renaming $n-1=\tilde{n}-2$ and noticing that $\tilde{n}\geq 4$) we get that
\begin{equation*}
\bigg|\lambda^{\frac{n-2}{2}}\sum_{j\neq 1}\frac{(y,\oxi_1-\oxi_j)}{|\oxi_1-\oxi_j|^n}\lambda\bigg|\leq \lambda^{\frac{n}{2}}|y|\sum_{j\neq 1}\frac{1}{|\oxi_1-\oxi_j|^{n-1}}\leq C\lambda^{\frac{n}{2}}k^{n-1}|y|\leq 
\begin{cases}
C\frac{|y|}{k}\mbox{ if }n\geq 4,\\
C\frac{|y|}{k(\ln k)^3}\mbox{ if }n=3,
\end{cases}
\end{equation*}
and \eqref{errPoint} follows for $|y|<\frac{\delta}{\lambda k}$.
\end{proof}

\section{The linear theory}\label{linear}

\noindent Consider the linear problem
 \begin{equation}\label{linRn}
 L_0(\varphi)=h(y)\quad\hbox{in}\quad \R^n,\qquad L_0(\varphi):=\Delta \varphi+p\gamma U^{p-1}\varphi.
 \end{equation}
 It is known that
$$\ker\{L_0\}=\mbox{span}\{Z_1,Z_2,\ldots,Z_{n+1}\},$$
with
\begin{equation}\label{Zdef}
Z_\alpha:=\partial_{y_\alpha}U,\;\; \alpha=1,\ldots,n,\quad Z_{n+1}:=y\cdot\nabla U+\frac{n-2}{2}U.
\end{equation}
Defining the norm
\begin{equation}\label{normStar}
\|\varphi\|_*:=\|(1+|y|^{n-2})\varphi\|_{L^\infty(\R^n)},
\end{equation}
in \cite[Lemma 3.1]{dPMPP} the following existence result is proved.
\begin{lemma}\label{existLin}
Assume $\frac{n}{2}<q<n$ in the definition of $\|\cdot\|_{**}$. Let $h$ be a function such that $\|h\|_{**}<+\infty$ and
$$\int_{\R^n}Z_\alpha h=0\;\;\mbox{ for all }\;\alpha=1,\ldots,n+1.$$
Then \eqref{linRn} has a unique solution $\varphi$ with $\|\varphi\|_*<+\infty$ such that
$$\int_{\R^n}U^{p-1}Z_\alpha \varphi=0\;\;\mbox{ for all }\alpha=1,\ldots,n+1,\quad \mbox{ and }\quad \|\varphi\|_*\leq C\|h\|_{**},$$
for some constant $C$ depending only on $q$ and $n$.
\end{lemma}

We will also need a priori estimates of the gradient of such solution. 
\begin{lemma}\label{estGradient}
Let $\varphi$ be the solution of \eqref{linRn} predicted by Lemma \ref{existLin} and assume $\|(1+|y|^{n+2})h\|_{L^\infty(\R^n)}<+\infty$. Then, there exists $C>0$ depending only on $n$ such that
$$\|\nabla\varphi\|_{L^\infty(\R^n)}\leq C(\|\varphi\|_*+\|(1+|y|^{n+2})h\|_{L^\infty(\R^n)}).$$
\end{lemma}
\begin{proof}
By standard elliptic estimates,
\begin{equation}\label{gradInt}
\|\nabla \varphi\|_{L^\infty(B_1)}\leq C\left(\|\varphi\|_{L^\infty(B_2)}+\|h\|_{L^\infty(B_2)}\right)\leq C\left(\|\varphi\|_*+\|(1+|y|^{n+2})h\|_{L^\infty(\R^n)}\right).
\end{equation}
Defining $\tilde{\varphi}(y):=|y|^{2-n}\varphi(|y|^{-2}y)$ it can be checked that
$$\Delta \tilde{\varphi}+p\gamma U^{p-1}\tilde{\varphi}=\tilde{h}\mbox{ in }\R^n\setminus\{0\},$$
with $\tilde{h}(y):=|y|^{-n-2}h(y |y|^{-2})$, and thus
$$\|\nabla \tilde{\varphi}\|_{L^\infty(B_1)}\leq C\left(\|\tilde{\varphi}\|_{L^\infty(B_2)}+\|\tilde{h}\|_{L^\infty(B_2)}\right).$$
Noticing that
$$\|\tilde{\varphi}\|_{L^\infty(B_2)}=\||y|^{n-2}\varphi\|_{L^\infty(\R^n\setminus B_{1/2})}\leq \|\varphi\|_*,$$
$$\|\tilde{h}\|_{L^\infty(B_2)}=\||y|^{n+2}h\|_{L^\infty(\R^n\setminus B_{1/2})}\leq \|(1+|y|^{n+2})h\|_{L^\infty(\R^n)},$$
we obtain
\begin{equation}\label{gradTilde}
\|\nabla \tilde{\varphi}\|_{L^\infty(B_1)}\leq C\left(\|\varphi\|_*+\|(1+|y|^{n+2})h\|_{L^\infty(\R^n)}\right).
\end{equation}
Writing $\varphi(y)=|y|^{-(n-2)}\tilde{\varphi}(|y|^{-2}y)$ it can be seen that
$$|\nabla\varphi(y)|\leq \frac{C}{|y|^{n-1}}|\tilde{\varphi}(|y|^{-2}y)|+\frac{C}{|y|^n}|\nabla\tilde{\varphi}(|y|^{-2}y)|,$$
and thus
\begin{equation}\label{gradExt}
\|\nabla \varphi\|_{L^\infty(\R^n\setminus B_1)}\leq C\left(\|\tilde{\varphi}\|_{L^\infty(B_1)}+\|\nabla \tilde{\varphi}\|_{L^\infty(B_1)}\right)
\end{equation}
Putting together \eqref{gradInt}, \eqref{gradTilde} and \eqref{gradExt} the estimate follows.
\end{proof}

\section{The gluing scheme}\label{gluing}
\noindent Our goal will be to find a solution of the form
\begin{equation}\label{ansatz1}u=\textbf{u}+\phi,\end{equation}
with $\phi$ a small function (in a sense that will be precised later). Thus, $u$ is a solution  of \eqref{prob} if and only if
\begin{equation}\label{probPhi}
\Delta\phi+p\gamma |\textbf{u}|^{p-1}\phi +E+\gamma N(\phi)=0,
\end{equation}
where
\begin{equation*}\begin{split}
N(\phi)&:=|\textbf{u}+\phi|^{p-1}(\textbf{u}+\phi)-|\textbf{u}|^{p-1}\textbf{u}-p|\textbf{u}|^{p-1}\phi,
\end{split}\end{equation*}
and $E$ was defined in \eqref{error}.

Let $\zeta(s)$ be a smooth function such that $\zeta(s)=1$ for $s<1$ and $\zeta(s)=0$ for $s>2$, and let $\delta >0$ be a fixed number independent of $k$. Let us define
\begin{equation}\label{ccut}
\overline{\zeta}_j(y):=\begin{cases}
\zeta(k \delta^{-1}|y|^{-2}|y-\bar \xi_j|y|^2|)\quad\hbox{ if }|y|>1,\\
\zeta(k \delta^{-1}|y-\bar \xi_j|)\quad\hbox{ if }|y|\leq 1,
\end{cases}\quad
\underline{\zeta}_j(y):=\overline{\zeta}_j(y_1,y_2,-y_3,\ldots,y_n)\end{equation}
for $j=1,\ldots,k$.
We notice that a function $\phi$ of the form
$$
\phi=\sum_{j=1}^k(\overline{\phi}_j+\underline{\phi}_j)+\psi
$$
is a solution of \eqref{probPhi} provided the functions $\overline{\phi}_j$, $\underline{\phi}_j$ and $\psi$ solve the following system of coupled non linear equations
\begin{equation}\label{system1}
\Delta \ophi_j+p\gamma |\textbf{u}|^{p-1}\ozeta_j\ophi_j+\ozeta_j\left[p\gamma |\textbf{u}|^{p-1}\psi+E+\gamma N(\phi)\right]=0, \; \;j=1,\ldots,k,
\end{equation}
\begin{equation}\label{system2}
\Delta \uphi_j+p\gamma |\textbf{u}|^{p-1}\uzeta_j\uphi_j+\uzeta_j\left[p\gamma |\textbf{u}|^{p-1}\psi+E+\gamma N(\phi)\right]=0, \; \;j=1,\ldots,k,
\end{equation}
\begin{equation}\label{system3}\begin{split}
\Delta\psi&+p\gamma U^{p-1}\psi+\left[p\gamma(|\textbf{u}|^{p-1}-U^{p-1})(1-\sum_{j=1}^k(\ozeta_j+\uzeta_j))\right.\\
&\left.+p\gamma U^{p-1}\sum_{j=1}^k(\ozeta_j+\uzeta_j)\right]\psi+ p\gamma |\textbf{u}|^{p-1}\sum_{j=1}^k(1-\ozeta_j)\ophi_j\\
&+p\gamma |\textbf{u}|^{p-1}\sum_{j=1}^k(1-\uzeta_j)\uphi_j+\left(1-\sum_{j=1}^k(\ozeta_j+\uzeta_j)\right)(E+\gamma N(\phi))=0.
\end{split}\end{equation}
Given our setting it is natural to ask for some symmetry properties on $\ophi_j$ and $\uphi_j$. In particular, denoting $\hat{y}:=(y_1,y_2)$ and $y':=(y_3,\ldots,y_n)$, we want them to satisfy
\begin{equation}\label{cond1}
\ophi_j(\hat{y},y')=\ophi_1(e^{2\pi\frac{ \textcolor{black}{(j-1)}}{k}i} \hat{y},y'),\qquad j=1,\ldots\,\textcolor{black}{k},
\end{equation}
where
\begin{equation}\begin{split}\label{cond2}
&\ophi_1(y_1,\ldots,y_\alpha,\ldots,y_n)=\ophi_1(y_1,\ldots,-y_\alpha,\ldots,y_n),\;\; \alpha=2,4,\ldots,n,\\
&\ophi_1(y)=|y|^{2-n}\ophi_1(|y|^{-2}y),
\end{split}\end{equation}
and
\begin{equation}\label{cond3}
\uphi_1(y)=\ophi_1(y_1,y_2,-y_3,\ldots,y_n).
\end{equation}
\begin{remark}\label{remSym}
The functions $\ophi_j$ and $\uphi_j$ are not even in the third coordinate separately but $\ophi_j+\uphi_j$ is so, that is,
$$(\ophi_j+\uphi_j)(y)=(\ophi_j+\uphi_j)(y_1,y_2,-y_3,\ldots,y_n).$$
Likewise, the functions $(\ozeta_j+\uzeta_j)$ and $(\ozeta_j\ophi_j+\uzeta_j\uphi_j)$ are even in the third coordinate.
\end{remark}
\noindent For $\rho>0$ small and fixed we assume in addition
\begin{equation}\label{cond4}
\|\overline{\ophi}_1\|_*\leq \rho,
\end{equation}
where $\overline{\ophi}_1(y):=\lambda^{\frac{n-2}{2}}\overline{\phi}_1(\oxi_1+\lambda y)$ and $\|\cdot\|_*$ is defined in \eqref{normStar}.

\begin{prop}\label{existPsi}
There exist constants $k_0$, $C$ and $\rho_0$ such that, for all $k\geq k_0$ and $\rho<\rho_0$, if $\ophi_j$ and $\uphi_j$ satisfy conditions \eqref{cond1}-\eqref{cond4} then there exists a unique solution $\psi=\Psi(\overline{\ophi}_1)$ of \eqref{system3} such that
\begin{equation}\label{sym1}
\psi(y_1,\ldots,y_\alpha,\ldots)=\psi(y_1,\ldots,-y_\alpha,\ldots),\;\;\alpha=3,\ldots,n,
\end{equation}
\begin{equation}\label{sym2}\psi(\hat{y},y')=\psi(e^{\frac{2\pi \textcolor{black}{(j-1)}}{k}i} \hat{y},y'),\qquad j=1,\ldots\,\textcolor{black}{k},
\end{equation}
\begin{equation}\label{sym3}\psi(y)=|y|^{2-n}\psi(|y|^{-2}y),
\end{equation}
and
\begin{equation*}\begin{split}
\|\psi\|_*&\leq C\left(\|\overline{\ophi}_1\|_*+k^{1-\frac{n}{q}}\right)\;\;\mbox{if }n\geq 4,\quad
\|\psi\|_*\leq C\left(\|\overline{\ophi}_1\|_*+(\ln k)^{-1}\right)\;\;\mbox{if }n=3.
\end{split}\end{equation*}
Furthermore, given two functions $\phi^1,\,\phi^2$ the operator $\Psi$ satisfies
$$\|\Psi(\phi^1)-\Psi(\phi^2)\|_*\leq C\|\phi^1-\phi^2\|_*.$$
\end{prop}

\begin{proof}
We prove the result by combining a linear theory with a fixed point argument as in \cite[Lemma 4.1]{dPMPP}. Indeed, consider first the problem
\begin{equation}\label{lin}
\Delta \psi +p\gamma U^{p-1}\psi=h,
\end{equation}
with $h$ satisfying \eqref{sym1}, \eqref{sym2}, $\|h\|_{**}<+\infty$ and
\begin{equation}\label{symH}
h(y)=|y|^{-n-2}h(|y|^{-2}y).
\end{equation}
Proceeding as in \cite[Lemma 4.1]{dPMPP} one can apply Lemma \ref{existLin} to conclude the existence of a unique bounded solution $\psi=T(h)$ of \eqref{lin} satisfying symmetries \eqref{sym1}-\eqref{sym3} and
$$\|\psi\|_*\leq C\|h\|_{**},$$
where $C$ is a positive constant depending only on $n$ and $q$.

Let us denote
\begin{equation}\begin{split}\label{V1V2}
V(y)&:=\underbrace{p\gamma(|\textbf{u}|^{p-1}-U^{p-1})\left(1-\sum_{j=1}^k(\ozeta_j+\uzeta_j)\right)}_{V_1(y)}+\underbrace{p\gamma U^{p-1}\sum_{j=1}^k(\ozeta_j+\uzeta_j)}_{V_2(y)},
\end{split}\end{equation}
and
\begin{equation}\label{M}
M(\psi):=\left(1-\sum_{j=1}^k(\ozeta_j+\uzeta_j)\right)(E+\gamma N(\phi)).
\end{equation}
Thus, in order to solve \eqref{system3} by a fixed point argument, we write
$$\psi=-T\left(V\psi+ p\gamma |\textbf{u}|^{p-1}\left(\sum_{j=1}^k(1-\ozeta_j)\ophi_j+\sum_{j=1}^k(1-\uzeta_j)\uphi_j\right)+M(\psi)\right)=:\mathcal{M}(\psi),$$
where $\psi\in X$, the space of continuous functions with $\|\cdot\|_*<+\infty$ and satisfying \eqref{sym1}-\eqref{sym3}. Pointing out Remark \ref{remSym} and the special symmetries of $\textbf{u}$ it can be checked that
$$V\psi+ p\gamma |\textbf{u}|^{p-1}\left(\sum_{j=1}^k(1-\ozeta_j)\ophi_j+\sum_{j=1}^k(1-\uzeta_j)\uphi_j\right)+M(\psi)$$
satisfies \eqref{sym1}, \eqref{sym2} and \eqref{symH} for every $\psi\in X$ and thus $\mathcal{M}(\psi)$ is well defined. Let us see that $\mathcal{M}$ is actually a contraction mapping in the $\|\cdot\|_*$ norm in a small ball around the origin in $X$. Proceeding as in \cite[Lemma 4.1]{dPMPP} we see that
\begin{equation}\label{estV}
\|V\psi\|_{**}\leq \begin{cases}
 Ck^{1-\frac{n}{q}}\|\psi\|_{*} \quad \mbox{ if }n\geq 4,\\
\frac{C}{\ln k}\|\psi\|_* \quad \mbox{ if }n=3,
\end{cases}
\|U^{p-1}\sum_{j=1}^k(\ozeta_j+\uzeta_j)\|_{**}\leq
\begin{cases}
Ck^{1-\frac{n}{q}}\|\oophi_1\|_* \quad \mbox{ if }n\geq 4,\\
\frac{C}{\ln k}\|\oophi_1\|_*  \quad \mbox{ if }n=3,
\end{cases}
\end{equation}
whenever $|y-\oxi_j|>\frac{\delta}{k}$, $|y-\uxi_j|>\frac{\delta}{k}$. Using Proposition \ref{Pest2} we get
\begin{equation}\label{estM}
\|M(\psi)\|_{**}\leq\begin{cases}
Ck^{1-\frac{n}{q}}(1+\|\oophi_1\|_*^2)+C\|\psi\|_*^2 \quad \mbox{ if }n\geq 4,\\
C\frac{1}{\ln k}(1+\|\oophi_1\|_*^2)+C\|\psi\|_*^2 \quad \mbox{ if }n=3.
\end{cases}
\end{equation}
Similarly, for $\psi_1$, $\psi_2$ such that $\|\psi_1\|_*<\rho$, $\|\psi_2\|_*<\rho$ it can be seen that
$$\|M(\psi_1)-M(\psi_2)\|_{**}\leq C\rho \|\psi_1-\psi_2\|_*.$$
This estimate, together with \eqref{estV}-\eqref{estM}, allows us to conclude that for $\rho$ small enough (independent of $k$) the operator $\mathcal{M}$ is a contraction map in the set of functions $\psi\in X$ with
$$\|\psi\|_*\leq
\begin{cases}
C(\|\oophi_1\|_*+k^{1-\frac{n}{q}}) \quad \mbox{ if }n\geq 4,\\
C(\|\oophi_1\|_*+(\ln k)^{-1}) \quad \mbox{ if }n=3.
\end{cases}$$
Then the lemma follows by a fixed point argument.
\end{proof}
We can also establish a uniform bound on the gradient of the function.
\begin{prop}\label{improvedEstimates}
Let $\psi$ be the solution of \eqref{system3} provided by Proposition \ref{existPsi}. Then there exists $C>0$, depending only on $n$, such that
\begin{equation}\label{estPsiGrad}
\|\nabla \psi\|_{L^\infty(\R^n)}\leq C.
\end{equation}
\end{prop}
\begin{proof}
The goal is to estimate the terms in \eqref{system3} to apply the a priori estimate in Lemma \ref{estGradient}. Consider $V_1$ and $V_2$ defined in \eqref{V1V2}. Thus,
\begin{equation}\begin{split}\label{psiEst11}
(1+|y|^{n+2})|V_1\psi|&\leq C\|\psi\|_*(1+|y|^4)U^{p-2}\sum_{j=1}^k\left(\frac{\lambda^{\frac{n-2}{2}}}{|y-\oxi_j|^{n-2}}+\frac{\lambda^{\frac{n-2}{2}}}{|y-\uxi_j|^{n-2}}\right)\leq C\|\psi\|_*,\\
&(1+|y|^{n+2})|V_2\psi|\leq \|\psi\|_*(1+|y|^{4})U^{p-1}\leq C\|\psi\|_*.
\end{split}\end{equation}
Analogously, noticing that $(1+|y|^{n+2})U^p\approx 1$, we have
$$
(1+|y|^{n+2})|\textbf{u}|^{p-1}\sum_{j=1}^k(1-\ozeta_j)|\ophi_j|
+ |\textbf{u}|^{p-1}\sum_{j=1}^k(1-\uzeta_j)|\uphi_j|\leq C\|\oophi_1\|_*,
$$
and
\begin{equation}\label{psiEst41}
\begin{split}
(1+|y|^{n+2})\left(1-\sum_{j=1}^k(\ozeta_j+\uzeta_j)\right)|N(\phi)|&\leq (1+|y|^{n+2})U^{p-2}\left(\bigg|\sum_{j=1}^k(\ophi_j+\uphi_j)\bigg|^2+|\psi|^2\right)\\
&\leq C(\|\oophi_1\|_*^2+\|\psi\|_*^2).
\end{split}\end{equation}
On the other hand
$$(1+|y|^{n+2})(1-\sum_{j=1}^k(\ozeta_j+\uzeta_j))|E|\leq C,$$
as a consequence of \eqref{errorFuera} and \eqref{sumFuera}. Using the estimates in Proposition \ref{existPsi} and Lemma \ref{estGradient} we conclude \eqref{estPsiGrad}.
\end{proof}
Let $\psi=\Psi(\overline{\ophi}_1)$ given by Proposition \ref{existPsi}. Notice that, thanks to the imposed conditions \eqref{cond1} and \eqref{cond3}, solving the systems \eqref{system1} and \eqref{system2} can be reduced to solving the equation for $\ophi_1$, that is,
$$\Delta \ophi_1+p\gamma |\textbf{u}|^{p-1}\ozeta_1\ophi_1+\ozeta_1\left[p\gamma |\textbf{u}|^{p-1}\Psi(\overline{\ophi}_1)+E+\gamma N(\phi)\right]=0\mbox{ in }\R^n,$$
or equivalently
\begin{equation}\label{eq1}
\Delta \ophi_1+p\gamma |\overline{U}_1|^{p-1}\ophi_1+\ozeta_1 E+\gamma\mathcal{N}(\overline{\phi}_1,\phi)=0,
\end{equation}
where
$$\mathcal{N}(\overline{\phi}_1,\phi):=p\left(|\textbf{u}|^{p-1}\ozeta_1-|\overline{U}_1|^{p-1}\right)\ophi_1+\ozeta_1\left[p|\textbf{u}|^{p-1}\Psi(\overline{\ophi}_1)+N(\phi)\right].$$
Denote
$$\overline{Z}_\alpha(y):=\lambda^{-\frac{n-2}{2}}Z_\alpha\left(\frac{y-\oxi_1}{\lambda}\right),\;\;\alpha=1,\ldots,n+1$$
where $Z_\alpha$ was defined in \eqref{Zdef}.
In order to solve \eqref{eq1} we will deal first with a projected linear version. Given a general function $\overline{h}$ we consider
\begin{equation}\label{projLin}
\Delta \ophi+p\gamma |\overline{U}_1|^{p-1}\ophi+\oh=c_3\oU_1^{p-1}\oZ_3+c_{n+1}\oU_1^{p-1}\oZ_{n+1},
\end{equation}
where
$$c_3:=\frac{\int_{\R^n}\oh\oZ_3}{\int_{\R^n}\oU_1^{p-1}\oZ^2_3},\qquad c_{n+1}:=\frac{\int_{\R^n}\oh\oZ_{n+1}}{\int_{\R^n}\oU_1^{p-1}\oZ^2_{n+1}}.$$

\begin{lemma}\label{existProjLinear}
Suppose that $\oh$ is even with respect to $y_2,y_4,\ldots,y_n$ and satisfies \eqref{symH}, and assume that $h(y):=\lambda^{\frac{n+2}{2}}\oh(\oxi_1+\lambda y)$ satisfies $\|h\|_{**}<+\infty$.

Then problem \eqref{projLin} has a unique solution $\ophi=\oT(\oh)$ that is even with respect to $y_2,y_4,\ldots,y_n$ and satisfies
$$\ophi(y)=|y|^{2-n}\ophi(|y|^{-2}y),$$
$$\int_{\R^n}\oophi U^{p-1}Z_{n+1}=0,\qquad \int_{\R^n}\oophi U^{p-1}Z_3=0,\qquad \|\oophi\|_*\leq C\|h\|_{**},$$
with $\oophi(y):=\lambda^{\frac{n-2}{2}}\ophi(\oxi_1+\lambda y)$.
\end{lemma}
\begin{proof}
Notice that, up to redefining $\oh$ as $\oh-c_3\oU_1^{p-1}\oZ_3-c_{n+1}\oU_1^{p-1}\oZ_{n+1},$
we can assume
$$\int_{\R^n}\oh\oZ_3=\int_{\R^n}\oh\oZ_{n+1}=0,$$
i.e., $c_3=c_{n+1}=0$ and thus equation \eqref{projLin} is equivalent to
$$\Delta \oophi+p\gamma|U|^{p-1}\oophi=-h \qquad \hbox{in }\mathbb{R}^n.$$
We want to apply \cite[Lemma 3.1]{dPMPP} to solve this problem, and therefore we need to prove that
$$\int_{\R^n} hZ_\alpha=0\,\mbox{ for all }\alpha=1,2,4,5,\ldots,n.$$
This follows straightforward for $\alpha=2,4,5,\ldots,n$ due to the evenness of $h$. The case $\alpha=1$ holds as a consequence of \eqref{symH} (see the proof of \cite[Lemma 4.2]{dPMPP}).
Then the result follows by \cite[Lemma 3.1]{dPMPP}.
\end{proof}
If instead of satisfying condition \eqref{symH} the function $h$ is even in all its coordinates we can prove a similar result (notice that in such case $c_3=0$).
\begin{lemma}\label{exisSim}
Suppose that $h(y):=\lambda^{\frac{n+2}{2}}\oh(\oxi_1+\lambda y)$ is even with respect to $y_\alpha$ for every $\alpha=1,\ldots, n$ and $\|h\|_{**}<+\infty$.
Then problem \eqref{projLin} has a unique solution $\ophi=\oT(\oh)$ that is even with respect to $y_\alpha$ for every $\alpha=1,\ldots,n$ and satisfies
$$\int_{\R^n}\oophi U^{p-1}Z_{n+1}=0,\qquad \|\oophi\|_*\leq C\|h\|_{**},$$
with $\oophi(y):=\lambda^{\frac{n-2}{2}}\ophi(\oxi_1+\lambda y)$.
\end{lemma}
\begin{proof}
The result follows as in the previous Lemma just by noticing that
$$\int_{\R^n}hZ_\alpha=0\mbox{ for all }\alpha=1,\ldots,n,$$
due to the evenness of $h$.
\end{proof}
As a consequence of these lemmas we are able to solve the projected version of \eqref{eq1}. Indeed, consider
\begin{equation}\label{eq2}
\Delta \ophi_1+p\gamma |\overline{U}_1|^{p-1}\ophi_1+\ozeta_1 E+\gamma\mathcal{N}(\overline{\phi}_1,\phi)=c_3\oU_1^{p-1}\oZ_3+c_{n+1}\oU_1^{p-1}\oZ_{n+1},
\end{equation}
with
\begin{equation}\label{c3cn1}
c_3:=\frac{\int_{\R^n}(\ozeta_1 E+\gamma\mathcal{N}(\overline{\phi}_1,\phi))\oZ_3}{\int_{\R^n}\oU_1^{p-1}\oZ^2_3},\qquad c_{n+1}:=\frac{\int_{\R^n}(\ozeta_1 E+\gamma\mathcal{N}(\overline{\phi}_1,\phi))\oZ_{n+1}}{\int_{\R^n}\oU_1^{p-1}\oZ^2_{n+1}}.
\end{equation}

\begin{prop}\label{existPhi1}
There exists a unique solution $\overline{\phi}_1=\overline{\phi}_1(\ell,t)$ of \eqref{eq2}, that satisfies
$$\|\oophi_1\|_*\leq C k^{-\frac{n}{q}}\mbox{ if }n\geq 4,\qquad \|\oophi_1\|_*\leq \frac{C}{k\ln k}\mbox{ if }n=3,$$
and
$$\|\overline{\overline{\mathcal{N}}}(\overline{\phi}_1,\phi)\|_{**}\leq Ck^{-\frac{2n}{q}}\mbox{ if }n\geq 4,\qquad \|\overline{\overline{\mathcal{N}}}(\overline{\phi}_1,\phi)\|_{**}\leq \frac{C}{(k\ln k)^2}\mbox{ if }n=3,$$
where $\oophi_1(y):=\lambda^\frac{n-2}{2}\ophi_1(\oxi_1+\lambda y)$ and $\overline{\overline{\mathcal{N}}}(\overline{\phi}_1,\phi)(y):=\lambda^\frac{n+2}{2}\mathcal{N}(\overline{\phi}_1,\phi)(\oxi_1+\lambda y)$. 
\end{prop}

\begin{proof}
We will solve \eqref{eq2} by means of a fixed point argument, writting
$$\ophi_1=\oT(\ozeta_1 E+\gamma\mathcal{N}(\overline{\phi}_1,\phi))=:\mathcal{M}(\phi_1),$$
where $\oT$ is the linear operator specified in Lemma \ref{existProjLinear}.
 To do so, we begin analyzing the nonlinear term, that can be decomposed as
$$\mathcal{N}(\overline{\phi}_1,\phi)=f_1+f_2+f_3+f_4,$$
where
$$f_1:=p\ozeta_1(|{\bf u}_2|^{p-1}-\oU_1^{p-1})\ophi_1,\quad f_2:=(\ozeta_1-1)\oU_1^{p-1}\ophi_1,$$
$$f_3:=\ozeta_1 p|{\bf u}_2|^{p-1}\Psi(\ophi_1),\quad f_4:=\ozeta_1N(\phi_1).$$
To estimate these terms we proceed in the same way as \cite[Proposition 4.1]{dPMPP}, so we just highlight the differences. Given a general function $f$, let us denote $\tilde{f}(y):=\lambda^\frac{n+2}{2}f(\oxi_1+\lambda y)$.
Assume first $n\geq 4$. Thus, noticing that
$$\sum_{j=1}^kU(y+\lambda^{-1}(\uxi_j-\oxi_1))\leq C\lambda^{n-2}\left(k^{n-2}\sum_{j=1}^k\frac{1}{j^{n-2}}+\frac{1}{(2\tau)^{n-2}}\right)\leq C\lambda^{n-2},$$
and using Proposition \ref{existPsi} one gets
\begin{equation}\label{f1}
|\tilde{f}_1(y)|\leq C\lambda^{\frac{n-2}{2}}U(y)^{p-2}|\oophi_1(y)|\;\;\hbox{ for }|y|<\frac{\delta}{\lambda k},\quad \|\tilde{f}_1\|_{**}\leq C \lambda ^{\frac{n}{2q}}\|\oophi_1\|_*,
\end{equation}
$$
|\tilde{f}_2(y)|\leq C U(y)^{p-1}|\oophi_1(y)|\;\;\hbox{ for }|y|>c\lambda^{-1/2},\quad\|\tilde{f}_2\|_{**}\leq C\lambda^{\frac{n}{2q}}\|\oophi_1\|_*,
$$
$$
|\tilde{f}_3(y)|\leq C U^{p-1}(y)\lambda^{\frac{n-2}{2}}|\psi(\oxi_1+\lambda y)|\;\;\hbox{ for }|y|>c\lambda^{-1/2},\quad \|\tilde{f}_3\|_{**}\leq C\lambda^{\frac{n}{2q}}(\|\oophi_1\|_*+k^{1-\frac{n}{q}}).
$$
Notice that
$$\tilde{N}(\phi)=|V_*+\hat{\phi}|^{p-1}(V_*+\hat{\phi})-|V_*|^{p-1}V_* -p|V_*|^{p-1}\hat{\phi},$$
where $\hat{\phi}(y):=\lambda^{\frac{n-2}{2}}\phi (\oxi_1+\lambda y)$ and
$$V_*(y):=-U(y)-\sum_{j\neq 1}U(y-\lambda^{-1}(\oxi_j-\oxi_1))-\sum_{j=1}^k U(y-\lambda^{-1}(\uxi_j-\oxi_1))+\lambda^\frac{n-2}{2}U(\oxi_1+\lambda y).$$
Hence
$$
|\tilde{f}_4(y)|\leq CU(y)^{p-2}(|\oophi_1(y)|^2+\lambda^{n-2}|\psi(\oxi_1+\lambda y)|^2)\;\;\hbox{ for }|y|<\frac{\delta}{\lambda k},\quad \|\tilde{f}_4\|_{**}\leq C\lambda^{\frac{n}{2q}}(\|\oophi_1\|_*+k^{1-\frac{n}{q}}).
$$
Finally, defining $f_5:=\ozeta_1E$ and using estimate \eqref{formufinal} we also have
$$\|\tilde{f}_5\|_{**}\leq C\lambda^{\frac{n}{2q}}.$$
In the case $n=3$ one has
\begin{equation}\label{ftilde1}
\|\tilde{f}_1\|_{**}\leq \frac{C}{k\ln k}\|\oophi_1\|_*, \quad \|\tilde{f}_2\|_{**}\leq \frac{C}{k\ln k}\|\oophi_1\|_*,\quad \|\tilde{f}_3\|_{**}\leq \frac{C}{k\ln k}\left(\|\oophi_1\|_*+\frac{1}{k\ln k}\right),
\end{equation}
\begin{equation}\label{ftilde2}
\|\tilde{f}_4\|_{**}\leq \frac{C}{k\ln k}\left(\|\oophi_1\|_*+\frac{1}{k\ln k}\right),\quad \|\tilde{f}_5\|_{**}\leq \frac{C}{k\ln k}.
\end{equation}
Applying Proposition \ref{Pest2}, estimates \eqref{f1}-\eqref{ftilde2}, Proposition \ref{existPsi} and Lemma \ref{existProjLinear}, we conclude that $\mathcal{M}$ is a contraction that maps functions $\overline{\ophi}$ with
$$\|\oophi\|_*\leq C k^{-\frac{n}{q}}\mbox{ if }n\geq 4,\qquad \|\oophi\|_*\leq \frac{C}{k\ln k}\mbox{ if }n=3,$$
into the same class of functions whenever $\frac{n}{2}<q<\frac{n}{2-\frac{2}{n-1}}$. Analogously, it can be proved the  Lipschitz character of the operators and thus, applying a fixed point argument, we conclude the proof.
\end{proof}

\begin{prop}\label{improvedEstimates2}
Let $\ophi_1$ be the solution of \eqref{eq2} provided by Proposition \ref{existPhi1}. Then there exists $C>0$, depending only on $n$, such that
$$|\oophi_1(y)|\leq C\frac{\lambda^{\frac{n-2}{2}}}{(1+|y|)^\alpha}\quad\mbox{ where }\begin{cases}
\alpha=2\;\;\mbox{ if }n\geq 5,\\
\alpha=1\;\;\mbox{ if }n=4,\\
0<\alpha<1\;\;\mbox{ if }n=3,
\end{cases}$$
where $\oophi_1(y):=\lambda^{\frac{n-2}{2}}\ophi_1(\oxi_1+\lambda y)$. 
\end{prop}
\begin{proof}
Denote $L_0(\phi):=\Delta \phi+p\gamma|U|^{p-1}\phi$. Thus, \eqref{eq2} can be written in the form
$$L_0(\oophi_1)+a(y)\oophi_1=g(y)+c_3U^{p-1}Z_3+c_{n+1}U^{p-1}Z_{n+1},$$
with $a(y):=\lambda^2p\gamma(|\textbf{u}|\ozeta_1-|\oU_1|^{p-1})(\oxi_1+\lambda y)$. Hence
$$|a(y)|\leq CU^{p-1}, \qquad |g(y)|\leq C\frac{\lambda^\frac{n-2}{2}}{(1+|y|)^4}.$$
Applying \cite[Lemma 3.2]{dPMPP} with $\nu=4$ for $n\geq 5$, $\nu=3$ for $n=4$ and $2<\nu<3$ for $n=3$ we get the desired estimates on $\oophi_1$. 
\end{proof}
To perform the reduction procedure in section \ref{thm} we will need more precise estimates on the pointwise behavior of $\ophi_1$, in particular on the part that will not be orthogonal to the kernel, whose size is smaller.
\begin{prop}\label{estSimPart}
Let $\ophi_1$ be the solution of \eqref{eq2} provided by Proposition \ref{existPhi1} and denote $\oophi_1(y):=\lambda^{\frac{n-2}{2}}\ophi_1(\oxi_1+\lambda y)$. Then there exists a decomposition 
$\oophi_1=\overline{\overline{\phi^s_1}}+\overline{\overline{\phi^*_1}}$ such that $\overline{\overline{\phi^s_1}}$ is even with respect to $y_\alpha$ for every $\alpha=1,\ldots,n$ and
\begin{equation}\label{estPhiNonSym}|\overline{\overline{\phi^*_1}}(y)|\leq 
\begin{cases}
C \frac{\lambda^{\frac{n-2}{2}}}{k}\frac{1}{1+|y|}\mbox{ if }n\geq 4,\\
C\frac{\lambda^{1/2}}{k(\ln k)^3}\frac{1}{1+|y|^\alpha}\quad 0<\alpha<1,\quad \mbox{ if }n=3.
\end{cases}.
\end{equation}
\end{prop}
\begin{proof}
The idea of the proof is to identify in the equation \eqref{eq2} the largest terms, that happen to be symmetric, and will produce a large but symmetric solution. The remaining terms will give a function $\overline{\overline{\phi_1^*}}$ non symmetric but smaller.
Denote
$$V^s:=\lambda^{\frac{n-2}{2}}\left[\sum_{j\neq 1}\frac{2^{\frac{n-2}{2}}\lambda^{\frac{n-2}{2}}}{|\oxi_1-\oxi_j|^{n-2}}+\sum_{j=1}^k\frac{2^{\frac{n-2}{2}}\lambda^{\frac{n-2}{2}}}{|\oxi_1-\uxi_j|^{n-2}}-U(\oxi_1)\right].$$
Consider $E^s$ given by \eqref{errorSim} and let us define
$$\lambda^{\frac{n-2}{2}}\mathcal{N}^s(\oxi_1+\lambda y):=\ozeta_1(\oxi_1+\lambda y)\left(\tilde{f}^s_1+\tilde{f}^s_3+\tilde{f}^s_4\right)+(\ozeta_1(\oxi_1+\lambda y)-1)\tilde{f}^s_2,$$
where 
$$\tilde{f}^s_1:=p\left[(U(y)+V^s)^{p-1}-U(y)^{p-1}\right]\overline{\overline{\phi_1^s}},\quad \tilde{f}^s_2:=U(y)^{p-1}\overline{\overline{\phi_1^s}},$$
$$\tilde{f}^s_3:=p\lambda^{\frac{n-2}{2}}(U(y)+V^s)^{p-1}\psi(\oxi_1),\quad \tilde{f}^s_4:=(U(y)+V^s)^{p-2}(\hat{\phi}^s)^2,$$
being 
$$\hat{\phi}^s:=\overline{\overline{\phi_1^s}}+\sum_{j\neq 1}\lambda^{\frac{n+2}{2}}\ophi_j^s(\oxi_1+\lambda y)+\sum_{j=1}^k\lambda^{\frac{n+2}{2}}\uphi_j^s(\oxi_1+\lambda y).$$
Notice that, if $\overline{\overline{\phi_1^s}}$ is even in all its coordinates, then $\mathcal{N}^s(\oxi_1+\lambda y)$ is also even and groups the largest terms of $\mathcal{N}(\oxi_1+\lambda y)$. Thus, proceeding as in the previous Proposition we can find a solution to
$$
\Delta \overline{\phi_1^s}+p\gamma |\overline{U}_1|^{p-1}\overline{\phi_1^s}+\ozeta_1 E^s+\gamma\mathcal{N}^s(\ophi_1,\phi)=c^s_{n+1}\oU_1^{p-1}\oZ_{n+1},
$$
with
\begin{equation*}
c^s_{n+1}:=\frac{\int_{\R^n}(\ozeta_1 E^s+\gamma\mathcal{N}^s(\ophi_1,\phi))\oZ_{n+1}}{\int_{\R^n}\oU_1^{p-1}\oZ^2_{n+1}},\quad \overline{\overline{\phi_1^s}}(y)=\lambda^{\frac{n-2}{2}}\overline{\phi^s_1}(\oxi_1+\lambda y),
\end{equation*}
by applying Lemma \ref{exisSim} to perform a fixed point argument in the set of functions $\overline{\overline{\phi^s}}$ which are even in all their coordinates and have size
$$\|\overline{\overline{\phi^s}}\|_*\leq Ck^{-\frac{n}{q}}\mbox{ if }n\geq 4,\quad \|\overline{\overline{\phi^s}}\|_*\leq \frac{C}{k\ln k}\mbox{ if }n= 4.$$
Furthermore, proceeding as in Proposition \ref{improvedEstimates2} we can conclude that
$$|\overline{\overline{\phi_1^s}}(y)|\leq C\frac{\lambda^{\frac{n-2}{2}}}{(1+|y|)^\alpha}\quad\mbox{ where }\begin{cases}
\alpha=2\;\;\mbox{ if }n\geq 5,\\
\alpha=1\;\;\mbox{ if }n=4,\\
0<\alpha<1\;\;\mbox{ if }n=3.
\end{cases}$$
Let us define $\overline{\phi_1^*}:=\ophi_1-\overline{\phi_1^s}$. Hence it solves
\begin{equation}\label{eqNonSim}
\Delta \overline{\phi_1^*}+p\gamma |\overline{U}_1|^{p-1}\overline{\phi_1^*}+\ozeta_1 E^*+\gamma\mathcal{N}^*(\ophi_1,\phi)=c_3\oU_1^{p-1}\oZ_3+c^*_{n+1}\oU_1^{p-1}\oZ_{n+1},
\end{equation}
where $E^*$ is defined in Proposition \ref{symErrInt},
$$ \mathcal{N}^*(\ophi_1,\phi):=\mathcal{N}(\ophi_1,\phi)-\mathcal{N}^s(\ophi_1,\phi),$$
and
\begin{equation*}
c_3:=\frac{\int_{\R^n}(\ozeta_1 E^*+\gamma\mathcal{N}^*(\ophi_1,\phi))\oZ_3}{\int_{\R^n}\oU_1^{p-1}\oZ^2_3},\quad c^*_{n+1}:=\frac{\int_{\R^n}(\ozeta_1 E^*+\gamma\mathcal{N}^*(\ophi_1,\phi))\oZ_{n+1}}{\int_{\R^n}\oU_1^{p-1}\oZ^2_{n+1}}=c_{n+1}-c^s_{n+1}.
\end{equation*}
The key point is that, without the previous symmetric part, the terms left are smaller and more precise estimates con be done, Indeed, denote
$$V(y):=\lambda^{\frac{n-2}{2}}\left[\sum_{j\neq 1}\oU_j(\oxi_1+\lambda y)+\sum_{j=1}^k\uU_j(\oxi_1+\lambda y)-U(\oxi_1+\lambda y)\right].$$
Thus we can write
$$\lambda^{\frac{n-2}{2}}\mathcal{N}^*(\oxi_1+\lambda y):=\ozeta_1(\oxi_1+\lambda y)\left(\tilde{f}^*_1+\tilde{f}^*_3+\tilde{f}^*_4\right)+(\ozeta_1(\oxi_1+\lambda y)-1)\tilde{f}^*_2,$$
where 
\begin{equation*}\begin{split}
\tilde{f}_1^*(y):=&\,p\left[(U(y)+V^s(y))^{p-1}-U^{p-1}(y)\right]\overline{\overline{\phi_1^*}}(y)\\
&+p\left[(U(y)+V(y))^{p-1}-(U(y)+V^s(y))^{p-1}\right]\oophi_1(y),
\end{split}\end{equation*}
$$\tilde{f}_2^*(y):=U(y)^{p-1}\overline{\overline{\phi_1^*}}(y),$$
$$\begin{aligned}
\tilde{f}_3^*(y):=&\,p\lambda^{\frac{n-2}{2}}\left(U(y)+V(y)\right)^{p-1}\lambda \nabla\psi(\eta)y\\
&+p\lambda^{\frac{n-2}{2}}\left[\left(U(y)+V(y)\right)^{p-1}-\left(U(y)+V^s(y)\right)^{p-1}\right]\psi(\oxi_1),
\end{aligned}$$
$$\tilde{f}^*_4(y):=\tilde{f}_4(y)-\tilde{f}^s_4(y),$$
with $\tilde{f}_4$ defined in Proposition \ref{existPhi1}. Noticing that
$$\bigg|\left[(U(y)+V(y))^{p-1}-(U(y)+V^s(y))^{p-1}\right]\oophi_1(y)\bigg|\leq CU^{p-1}(y)|V(y)-V^s|\|\oophi_1\|_*$$
and
$$|V(y)-V^s|\leq C \frac{\lambda^{\frac{n-2}{2}}|y|}{k}\mbox{ if }n\geq 4,\qquad   |V(y)-V^s|\leq C\frac{\lambda^{1/2}|y|}{k(\ln k)^3}\mbox{ if }n=3,$$
it can be proved that 
\begin{equation*}
\bigg|\left[(U(y)+V(y))^{p-1}-(U(y)+V^s(y))^{p-1}\right]\oophi_1(y)\bigg|\leq\begin{cases}
C \frac{\lambda^{\frac{n-2}{2}}}{k}\frac{1}{1+|y|^3}\mbox{ if }n\geq 4,\\
C\frac{\lambda^{1/2}}{k(\ln k)^3}\frac{1}{1+|y|^3}\mbox{ if }n=3,
\end{cases}
\end{equation*}
and the same bound can be proved for $\tilde{f}_3^*$ and $\tilde{f}_4^*$ by using the estimates in Proposition \ref{existPsi} and Proposition \ref{improvedEstimates}. Using this together with Proposition \ref{symErrInt} we can proceed as in the proof of Proposition \ref{improvedEstimates2}  to estimate the size of $|\overline{\overline{\phi_1^*}}|$. That is, we can rewrite problem \eqref{eqNonSim} as
$$L_0(\overline{\overline{\phi_1^*}})+a^*(y)\overline{\overline{\phi_1^*}}=g^*(y)+c_3U^{p-1}Z_3+c_{n+1}^*U^{p-1}Z_{n+1},$$
where
$$L_0(\phi):=\Delta \phi+p\gamma|U|^{p-1}\phi,\qquad |a^*(y)|\leq CU(y)^{p-1},$$
and
$$|g^*(y)|\leq \begin{cases}
C \frac{\lambda^{\frac{n-2}{2}}}{k}\frac{1}{1+|y|^3}\mbox{ if }n\geq 4,\\
C\frac{\lambda^{1/2}}{k(\ln k)^3}\frac{1}{1+|y|^3}\mbox{ if }n=3.
\end{cases}.$$
Applying again \cite[Lemma 3.2]{dPMPP} with $\nu=3$ for $n\geq 4$ and $2<\nu<3$ for $n=3$ we obtain \eqref{estPhiNonSym}.
\end{proof}
Likewise, we will need accurate estimates on the non symmetric part of $\psi$.
\begin{prop}\label{symEstPsi}
Let $\psi$ be the solution of \eqref{system3} provided by Proposition \ref{existPsi}. Then
$$\psi(\oxi_1+\lambda y)=\psi^s(y)+\psi^*(y), \qquad y\in B(0,\frac{\delta}{\lambda k}),$$
where $\psi^s$ is even with respect to $y_3$ and
$$|\psi^*(y)|\leq C \left(\|\oophi_1\|_*+\|\psi\|_*+\tau o_k(1)\right)\lambda |y|(1+|y|),$$
where $o_k(1)$ is a function that goes to 0 when $k\to \infty$.
\end{prop}
\begin{proof}
Since $\psi$ is a solution of \eqref{system3} we can write, making the convolution with the fundamental solution of the Laplace equation,
\begin{equation*}\begin{split}
\psi(\oxi_1+\lambda y)=c_n\int_{\R^n}&\frac{1}{|\oxi_1+\lambda y-x|^{n-2}} W(\psi)(x) \,dx, \quad y\in B(0,\frac{\delta}{\lambda k}), 
\end{split}\end{equation*}
$$W(\psi):=V\psi+M(\psi)+p\gamma|{\bf{u}}|^{p-1}\sum_{j=1}^k((1-\ozeta_j)\ophi_j+(1-\uzeta_j)\uphi_j)-p\gamma U^{p-1}\psi,$$
where $c_n$ is a constant depending only on the dimension and $V$, $M$ were defined in \eqref{V1V2} and \eqref{M} respectively. Furthermore,
\begin{equation*}\begin{split}
\frac{1}{|\oxi_1+\lambda y-x|^{n-2}}=:A(x,y)+B_1(x,y)+B_2(x,y),
\end{split}\end{equation*}
where
$$A(x,y):=\frac{1}{|x-\oxi_1|^{n-2}}\left[1-\left(\frac{n-2}{2}\right)\frac{\lambda^2 y^2+2\lambda\sum_{i=1, i\neq 3}^n y_i(x-\oxi_1)_i}{|x-\oxi_1|^2}\right],$$
$$B_1(x,y):=-(n-2)\lambda y_3\frac{x_3-\tau}{|x-\oxi_1|^n},\qquad B_2(x,y):=O\left(\left(\frac{(\lambda^2 y^2+2\lambda (y,x-\oxi_1))^2}{|x-\oxi_1|^{n+2}}\right)\right).$$
Notice that $A(x,y)$ is even with respect to $y_3$ and thus we can define
$$\psi^s(y):=c_n\int_{\R^n}A(x,y) W(\psi)(x)\,dx,$$
that inherits this symmetry. Therefore we have to estimate
$$\psi^*(y):=c_n\int_{\R^n}(B_1(x,y)+B_2(x,y)) W(\psi)(x)\,dx.$$
Since $|W(\psi)(x)|\leq \frac{C}{(1+|x|)^4}$
and $|y|<c\lambda^{-1/2}$ it easily follows that
\begin{equation}\label{estB2}
\bigg|\int_{\R^n}B_2(x,y) W(\psi)(x)\,dx\bigg| \leq  \lambda \tau |y|^2 o_k(1).
\end{equation}
To estimate the term with $B_1(x,y)$ notice first that
\begin{equation}\label{psiEst1}
\bigg|\int_{\R^n}B_1(x,y)U^{p-1}\psi\,dx\bigg|\leq C\lambda |y|\int_{\R^n}\frac{1}{|x-\oxi_1|^{n-1}}\frac{1}{(1+|x|)^4}|\psi(x)|\,dx \leq C\lambda |y|\|\psi\|_*
\end{equation}
and likewise 
\begin{equation}\label{psiEst2}
\bigg|\int_{\R^n}B_1(x,y)V \psi \,dx \bigg|\leq C\lambda |y|\|\psi\|_*.
\end{equation}
Observing that
$$|\ophi_j(y)|\leq C\|\oophi_1\|_*\frac{\lambda^{\frac{n-2}{2}}}{|y-\oxi_j|^{n-2}},\qquad |\uphi_j(y)|\leq C\|\oophi_1\|_*\frac{\lambda^{\frac{n-2}{2}}}{|y-\uxi_j|^{n-2}},$$
we get
\begin{equation}\label{psiEst3}
\bigg|\int_{\R^n}B_1(x,y)|{\bf{u}}|^{p-1}\sum_{j=1}^k((1-\ozeta_j)\ophi_j+(1-\uzeta_j)\uphi_j)\bigg|\leq C \lambda |y| \|\oophi_1\|_*,
\end{equation}
\begin{equation}\label{psiEst4}
\bigg|\int_{\R^n}B_1(x,y)\left(1-\sum_{j=1}^k(\ozeta_j+\uzeta_j)\right)N(\phi)\bigg|\leq C \lambda |y| (\|\oophi_1\|_*+\|\psi\|_*),
\end{equation}
and thus the only term left is the one concerning the error, namely
$$\int_{\R^n}B_1(x,y)\left(1-\sum_{j=1}^k(\ozeta_j+\uzeta_j)\right)E\,dx=\underbrace{\int_{\R^n}B_1(x,y)\left(1-\sum_{j=1}^k\ozeta_j\right)E\,dx}_{B_{11}}-\underbrace{\sum_{j=1}^k\int_{\R^n}B_1(x,y)\uzeta_j E\,dx}_{B_{12}}.$$
Let $R>0$ large. Thus it can be seen that
$$B_{11}=\int_{B(\oxi_1,R)}B_1(x,y)(1-\sum_{j=1}^k\ozeta_j)E\,dx+O(|y| k^{-(n-1)}).$$
The desired estimate will follow by noticing that the largest terms of the error are orthogonal to $B_1(x,y)$. To see this, we write the error as
\begin{equation*}\begin{split}
\gamma^{-1}E\simeq &\,pU^{p-1}\left(\sum_{j=1}^k\oU_j+\sum_{j=1}^k\uU_j\right)-\sum_{j=1}^k\oU_j^p-\sum_{j=1}^k\uU_j^p\\
=&\,pU^{p-1}(\oxi_1)\sum_{j=1}^k\oU_j-\sum_{j=1}^k\oU_j^p+pU^{p-1}(\oxi_1)\sum_{j=1}^k\uU_j-\sum_{j=1}^k\uU_j^p\\
&\,+p(U^{p-1}-U^{p-1}(\oxi_1))\sum_{j=1}^k\oU_j+p(U^{p-1}-U^{p-1}(\oxi_1))\sum_{j=1}^k\uU_j.
\end{split}\end{equation*}
Notice first that, since $B_1(x,y)$ is odd with respect to the hyperplane $x_3=\tau$, there holds
$$\int_{B(\oxi_1,R)}B_1(x,y)(1-\sum_{j=1}^k\ozeta_j)\left(pU^{p-1}(\oxi_1)\sum_{j=1}^k\oU_j-\sum_{j=1}^k\oU_j^p\right)\,dx=0.$$
On the other hand, writting
\begin{equation*}\begin{split}\label{uUj}
\uU_j(x)=&\,\frac{\lambda^\frac{n-2}{2}}{(\lambda^2+|\oxi_1-\uxi_j|^2+|x-\oxi_1|^2)^{\frac{n-2}{2}}}\left[1-(n-2)\frac{(x-\oxi_1,\oxi_1-\uxi_j)}{\lambda^2+|\oxi_1-\uxi_j|^2+|x-\oxi_1|^2}\right.\\
&\,\left.+O\left(\left(\frac{(x-\oxi_1,\oxi_1-\uxi_j)}{\lambda^2+|\oxi_1-\uxi_j|^2+|x-\oxi_1|^2}\right)^2\right)\right],
\end{split}\end{equation*}
and applying again the evenness of its main terms with respect to $x_3=\tau$ we have that
\begin{equation*}\begin{split}
\bigg|\int_{B(\oxi_1,R)}&B_1(x,y)\left(1-\sum_{j=1}^k\ozeta_j\right)pU^{p-1}(\oxi_1)\sum_{j=1}^k\uU_j\,dx\bigg|\\
\leq &\,C\lambda^{1+\frac{n-2}{2}}|y|\int_{B(\oxi_1,R)}\frac{|x_3-\tau|}{|x-\oxi_1|^n}\sum_{j=1}^k\frac{\tau|x_3-\tau|}{(\lambda^2+|\oxi_1-\uxi_j|^2+|x-\oxi_1|^2)^{\frac{n}{2}}}\,dx\\
&\,+C\lambda^{1+\frac{n-2}{2}}|y|\int_{B(\oxi_1,R)}\frac{|x_3-\tau|}{|x-\oxi_1|^n}\sum_{j=1}^k\frac{|x-\oxi_1|^2|\oxi_1-\uxi_j|^2}{(\lambda^2+|\oxi_1-\uxi_j|^2+|x-\oxi_1|^2)^{\frac{n+2}{2}}}\,dx\\
\leq&\, C\lambda^{1+\frac{n-2}{2}}|y|\int_{B(\oxi_1,R)}\frac{1}{|x-\oxi_1|^{n-1}}\sum_{j=1}^k\left(\frac{\tau}{|\oxi_1-\uxi_j|^{n-1}}+\frac{1}{|\oxi_1-\uxi_j|^{n-2}}\right)\,dx\\
=&\, \lambda\tau |y| o_k(1),
\end{split}\end{equation*}
where in the last inequality we have used \eqref{due}. Proceeding analogously with the other terms it can be concluded that 
$$\int_{B(\oxi_1,R)}B_1(x,y)\left(1-\sum_{j=1}^k\ozeta_j\right)E\,dx=\lambda\tau o_k(1)|y|,$$
and therefore
\begin{equation}\label{B11}
B_{11}=\lambda\tau o_k(1)|y|.
\end{equation}
Likewise,
\begin{equation}\begin{split}\label{B12}
|B_{12}|\leq C\lambda |y|\sum_{j=1}^k\int_{B(\uxi_j,\frac{\delta}{k})}\frac{|x_3-\tau|}{|\oxi_1-\uxi_j|^n}\sum_{i=1}^k\frac{\lambda^{\frac{n-2}{2}}}{|\oxi_1-\uxi_j|^{n-2}}\,dx=\lambda\tau o_k(1)|y|.
\end{split}\end{equation}
Putting together \eqref{estB2}-\eqref{psiEst4} with \eqref{B11} and \eqref{B12} the result follows.
\end{proof}

\section{Proof of Theorem \ref{theo}}\label{thm}
\noindent The goal of this section is to find positive parameters $\ell$ and $t$ that enter in the definition of $\lambda$ and $\tau$ in \eqref{par1} and are independent of $k$, in such a way that  $c_3$ and $c_{n+1}$ (defined in \eqref{c3cn1}) vanish. In fact, if such a choice is possible, then the solution $\ophi_1$ found in Proposition \ref{existPhi1} solves \eqref{eq1} and thus, applying Proposition \ref{existPsi} and \eqref{cond1}-\eqref{cond3}, we can conclude that $\phi$ solves \eqref{probPhi} and therefore
$$u=\textbf{u}+\phi$$
is a solution of \eqref{prob}.

Thus, we want to prove the existence of  $\ell$ and $t$ so that
$$\oc_3(\ell, t):=\int_{\R^n}(\ozeta_1 E+\gamma\mathcal{N}(\ophi_1,\phi))\oZ_3=0,\;\; \oc_{n+1}(\ell, t):=\int_{\R^n}(\ozeta_1 E+\gamma\mathcal{N}(\ophi_1,\phi))\oZ_{n+1}=0.$$
It is worth pointing out that, due to some symmetry, the main order term of $\oc_3$ vanishes, what makes necessary an expansion of $\oc_3$ at lower order. This is usually a delicate issue and it requires sharper estimates on the non linear term, that is, a finer control on the size of the terms $\psi$ and $\ophi_1$ in the spirit of Proposition \ref{improvedEstimates2}. However in this case this type of estimates are not enough, since they do not produce a non linear term sufficiently small. We need to identify a precise decomposition of  $\ophi_1$ and $\psi$ in one symmetric part whose contribution to the computation of $\oc_3$ is zero, and a smaller non-symmetric part (see Claim 6). This decompositions were developed in Proposition \ref{estSimPart} and Proposition \ref{symEstPsi}.

What we obtain at the end is that, for $n\geq 4$,
\begin{equation}\begin{split}\label{systLT}
&\oc_{3}(\ell, t)=D_n\frac{t\ell^{\frac{n}{n-2}}}{k^{n+1-\frac{2}{n-1}}}\left[d_n\frac{\ell}{t^{n-1}}-1\right]+\frac{1}{k^{\alpha}}\Theta_{k}(\ell,t),\\
&\oc_{n+1}(\ell, t)=E_n\frac{\ell}{k^{n-2}}\left[e_n\ell^2 -1\right]+\frac{1}{k^{\beta}}\Theta_{k}(\ell,t),
\end{split}\end{equation}
where
$$\alpha>n+1-\frac{2}{n-1},\qquad \beta>n-2,$$
and
\begin{equation}\begin{split}\label{systLTn3}
&\oc_3(\ell,t)=F\frac{\ell^3 t}{\sqrt{\ln k}(k\ln k)^3}\left[f\frac{\ell}{t^2}-1\right]+\frac{1}{k^3(\ln k)^4}\Theta_k(\ell,t),\\
&\oc_4(\ell,t)=G\frac{\ell}{k\ln k}\left[g\ell^2-1\right]+\frac{1}{k\ln k}\frac{\ln\left(\frac{2\pi}{\sqrt{\ln k}}\right)}{\ln k}\Theta_k(\ell,t),
\end{split}\end{equation}
 for $n=3$, where $D_n, d_n, E_n, e_n, F, f, G, g$ are fixed positive numbers (depending only on $n$) and $\Theta_{k}(\ell,t)$ is a generic function, smooth on its variables, and uniformly bounded as $k\rightarrow \infty$. Hence, by a fixed point argument we can conclude the existence of $\ell$ and $t$ such that
\begin{equation}\label{cero}
\oc_3(\ell, t)=\oc_{n+1}(\ell, t)=0.
\end{equation}
By simplicity we detail the argument in the case of \eqref{systLTn3}. With abuse of notation on the function $\Theta$, that always stands for a generic function smooth on its variables and uniformly bounded as $k\rightarrow \infty$, \eqref{cero} is equivalent to
\begin{equation*}\begin{split}
&f\ell-t^2+o_k(1)\Theta_k(\ell,t)=0,\\
&g\ell^2-1+o_k(1)\Theta_k(\ell,t)=0.
\end{split}\end{equation*}
Defining $\rho:=t^2$ and $\eta:=\ell^2$ we can rewrite the system as
\begin{equation*}\begin{split}
&\rho=f\eta^{1/2}+o_k(1)\Theta_k(\eta,\rho),\\
&\eta=\frac{1}{g}+o_k(1)\Theta_k(\eta,\rho).
\end{split}\end{equation*}
Suppose $0<\rho\leq C$ fixed. Hence the second equation can be expressed as
$$\eta=F_\rho(\eta),\quad \mbox{ with }F_a(s):=\frac{1}{g}+o_k(1)\Theta_k(s,a).$$ 
Consider the set $X:=\{\eta\in \R:\,0<\eta\leq \frac{2}{g}\}$. Using the smoothness of $\Theta_k$ it is easy to see that $F_\rho$ maps $X$ into itself and that it is a contraction for $k$ large enough. Thus, for any fixed $\rho$ there exists a fixed point $\eta_\rho\in X$ such that $\eta_\rho=F_\rho(\eta_\rho)$.

Replacing on the first equation this translates into
$$\rho=f (F_\rho(\eta_\rho))^{1/2}+o_k(1)\Theta_k(\eta_\rho,\rho)=\frac{f}{g^{1/2}}+o_k(1)\Theta_k(\eta_\rho,\rho)=:G(\rho).$$
Considering the set $Y:=\{\rho\in\R:0<\rho\leq \frac{2f}{g^{1/2}}\}$ and using the smoothness of $\Theta$ it can be checked that $G$ is a contraction that maps the set $Y$ into itself, and therefore we conclude the existence of a fixed point $\rho = G(\rho)$, what concludes the argument.

The rest of the section is devoted to prove \eqref{systLT} and \eqref{systLTn3}.
Let us consider first the case of $\oc_{n+1}$, that follows analogously to \cite{dPMPP}. We write
$$\oc_{n+1}(\ell, t)=\int_{\R^n}E\oZ_{n+1}-\int_{\R^{n}}(1-\ozeta_1)E\oZ_{n+1}+\gamma\int_{\R^n}\mathcal{N}(\ophi_1,\phi)\oZ_{n+1}=0.$$
Thus, for $\ell$ and $t$ as in \eqref{par1} we have:

\noindent {\bf Claim 1:}
$$ \int_{\R^n}E\oZ_{n+1}=
\begin{cases}E_n\frac{\ell}{k^{n-2}}\left[e_n\ell^2 -1\right]+\frac{1}{k^{n-2+2\frac{n-3}{n-1}}}\Theta_k(\ell,t) \;\;\mbox{ if }n\geq 4,\\
G\frac{\ell}{k\ln k}\left[g\ell^2-1\right]+\frac{1}{k\ln k}\frac{\ln\left(\frac{2\pi}{\sqrt{\ln k}}\right)}{\ln k}\Theta_k(\ell,t)\;\;\mbox{ if }n=3.
\end{cases}$$

\noindent {\bf Claim 2:}
$$\int_{\R^{n}}(1-\ozeta_1)E\oZ_{n+1}=
\begin{cases}\frac{1}{k^{n-1}}\Theta_k(\ell,t) \;\;\mbox{ if }n\geq 4,\\
\frac{1}{(k\ln k)^2}\Theta_k(\ell,t)\;\;\mbox{ if }n=3.
\end{cases}$$

\noindent {\bf Claim 3:}
$$\int_{\R^n}\mathcal{N}(\ophi_1,\phi)\oZ_{n+1}=
\begin{cases}
\frac{1}{k^{n+\frac{n}{q}-3}}\Theta_k(\ell,t) \;\;\mbox{ if }n\geq 4,\\
\frac{1}{(k\ln k)^2}\Theta_k(\ell,t)\;\;\mbox{ if }n=3.
\end{cases}$$
\noindent Notice that these claims together give the second equation in \eqref{systLT} and \eqref{systLTn3}.

\noindent {\bf Proof of Claim 1.} We decompose
\begin{equation}\label{three}
\int_{\mathbb{R}^n}E\oZ_{n+1}=\int_{B(\oxi_1,\frac{\delta}{k})}E\oZ_{n+1}+\int_{Ext}E\oZ_{n+1}+\sum_{j\neq 1}\int_{B(\oxi_j,\frac{\delta}{k})}E\oZ_{n+1}+\sum_{j= 1}^k\int_{B(\uxi_j,\frac{\delta}{k})}E\oZ_{n+1},
\end{equation}
where $\delta$ is a positive constant independent of $k$ and
$$Ext:=\{\cap_{j=1}^k\{|y-\oxi_j|>\frac{\delta}{k}\}\}\cap\{\cap_{j=1}^k\{|y-\uxi_j|>\frac{\delta}{k}\}\}.$$
Denoting
$$V(y):=\lambda^{\frac{n-2}{2}}\left[\sum_{j\neq 1}\oU_j(\lambda y+\oxi_1)-\sum_{j=1}^k\uU_j(\lambda y+\oxi_1)-U(\lambda y+\oxi_1)\right]$$
we have, for some $s\in(0,1)$,
\begin{equation}\label{mainError}
\begin{split}
&\gamma^{-1}\int_{B(\oxi_1,\frac{\delta}{k})}E\oZ_{n+1}=\lambda^\frac{n+2}{2}\int_{B(0,\frac{\delta}{\lambda k})}E(\lambda y+\oxi_1)Z_{n+1}(y)\\
&\qquad=p\sum_{j\neq 1}\lambda^{\frac{n-2}{2}}\int_{B(0,\frac{\delta}{\lambda k})}U^{p-1}\oU_j(\lambda y+\oxi_1)Z_{n+1}+p\sum_{j= 1}^k\lambda^{\frac{n-2}{2}}\int_{B(0,\frac{\delta}{\lambda k})}U^{p-1}\uU_j(\lambda y+\oxi_1)Z_{n+1}\\
&\qquad\;\;-p\lambda^{\frac{n-2}{2}}\int_{B(0,\frac{\delta}{\lambda k})}U^{p-1}U(\lambda y+\oxi_1)Z_{n+1}+p\int_{B(0,\frac{\delta}{\lambda k})}\left[(U+sV)^{p-1}-U^{p-1}\right]VZ_{n+1}\\
&\qquad\;\;+\sum_{j\neq 1}\lambda^{\frac{n+2}{2}}\int_{B(0,\frac{\delta}{\lambda k})}\oU_j^p(\lambda y+\oxi_1)Z_{n+1}+\sum_{j= 1}^k\lambda^{\frac{n+2}{2}}\int_{B(0,\frac{\delta}{\lambda k})}\uU_j^p(\lambda y+\oxi_1)Z_{n+1}\\
&\qquad\;\;-\lambda^{\frac{n+2}{2}}\int_{B(0,\frac{\delta}{\lambda k})}U^p(\lambda y+\oxi_1)Z_{n+1}.
\end{split} \end{equation}
Thus, defining
$I_1:=\int_{\R^n}U^{p-1}Z_{n+1},$
from \eqref{e0} follows that
$$
\lambda^{\frac{n-2}{2}}\int_{B(0,\frac{\delta}{\lambda k})}U^{p-1}\oU_j(\lambda y+\oxi_1)Z_{n+1}=\frac{2^{\frac{n-2}{2}}\lambda^{n-2}I_1}{|\oxi_1-\oxi_j|^{n-2}}\left(1+\frac{\lambda^2}{|\oxi_1-\oxi_j|^2}\Theta_k(\ell,t)\right),
$$
$$
\lambda^{\frac{n-2}{2}}\int_{B(0,\frac{\delta}{\lambda k})}U^{p-1}\uU_j(\lambda y+\oxi_1)Z_{n+1}=\frac{2^{\frac{n-2}{2}}\lambda^{n-2}I_1}{|\oxi_1-\uxi_j|^{n-2}}\left(1+\frac{\lambda^2}{|\oxi_1-\uxi_j|^2}\Theta_k(\ell,t)\right),
$$
and
$$
\lambda^{\frac{n-2}{2}}\int_{B(0,\frac{\delta}{\lambda k})}U^{p-1}U(\lambda y+\oxi_1)Z_{n+1}=\lambda^{\frac{n-2}{2}}U(\oxi_1)I_1\left(1+\frac{\lambda^2}{1+|\oxi_1|^2}\Theta_k(\ell,t)\right),
$$
which are the main order terms in \eqref{mainError}. Indeed,
$$\begin{aligned}
\bigg|\sum_{j\neq 1}\lambda^{\frac{n+2}{2}}&\int_{B(0,\frac{\delta}{\lambda k})}\oU_j^p(\lambda y+\oxi_1)Z_{n+1}\bigg|\leq C\sum_{j\neq 1}\frac{\lambda^{n+2}}{|{\oxi}_j-{\oxi}_1|^{n+2}}\int_{B(0,\frac{\delta}{\lambda k})}\frac{1}{(1+|y|)^{n-2}}\\
&\leq C
\begin{cases}
(\lambda k)^{-2}\sum_{j\neq 1}\frac{\lambda^{n+2}}{|{\oxi}_j-{\oxi}_1|^{n+2}}\;\;\mbox{ if }n\geq 4,\\
|\ln (\lambda k)|\sum_{j\neq 1}\frac{\lambda^{n+2}}{|{\oxi}_j-{\oxi}_1|^{n+2}}\;\;\mbox{ if }n=3,
\end{cases}
\end{aligned}$$
$$
\begin{aligned}
\bigg|\sum_{j=1}^k\lambda^{\frac{n+2}{2}}&\int_{B(0,\frac{\delta}{\lambda k})}\uU_j^p(\lambda y+\uxi_1)Z_{n+1}\bigg|\leq C\sum_{j\neq 1}\frac{\lambda^{n+2}}{|{\uxi}_j-{\oxi}_1|^{n+2}}\int_{B(0,\frac{\delta}{\lambda k})}\frac{1}{(1+|y|)^{n-2}}\\
&\leq C
\begin{cases}
(\lambda k)^{-2}\sum_{j\neq 1}\frac{\lambda^{n+2}}{|{\uxi}_j-{\oxi}_1|^{n+2}}\;\;\mbox{ if }n\geq 4,\\
|\ln (\lambda k)|\sum_{j\neq 1}\frac{\lambda^{n+2}}{|{\uxi}_j-{\oxi}_1|^{n+2}}\;\;\mbox{ if }n=3,
\end{cases}
\end{aligned}$$
and
$$\begin{aligned}
\bigg|\lambda^{\frac{n+2}{2}}\int_{B(0,\frac{\delta}{\lambda k})}&U^p(\lambda y+\oxi_1)Z_{n+1}\,dy\bigg|\leq C\lambda^{\frac{n+2}{2}}\int_{B(0,\frac{\delta}{\lambda k})}\frac{1}{(1+|y|)^{n-2}}\\
&\leq C
\begin{cases}
\lambda^{\frac{n-2}{2}}k^{-2}\;\;\mbox{ if }n\geq 4,\\
\lambda^{\frac{n+2}{2}}|\ln (\lambda k)|\;\;\mbox{ if }n=3.
\end{cases}
\end{aligned}$$
Finally these three estimates, together with the mean value theorem, also imply
$$\begin{aligned}
\bigg|p\int_{B(0,\frac{\delta}{\lambda k})}&\left[(U+sV)^{p-1}-U^{p-1}\right]VZ_{n+1}\bigg|\\
&\leq C
\begin{cases}
\lambda^{\frac{n-2}{2}}k^{-2}+(\lambda k)^{-2}(\lambda k)^{n+2}\;\;\mbox{ if }n\geq 4,\\
\lambda^{\frac{n+2}{2}}|\ln (\lambda k)|+(\lambda k)^{n+2}|\ln (\lambda k)|\;\;\mbox{ if }n=3.
\end{cases}
\end{aligned}$$
Proceeding like in \cite[Proof of Claim 2]{dPMPP} we obtain the estimates of the other terms in \eqref{three}, that is,
\begin{equation}\begin{split}\label{Zfuera}
\bigg|\int_{Ext}E\oZ_{n+1}\bigg|\leq C\begin{cases}
\frac{1}{k^{n-1}}\;\;\mbox{ if }n\geq 4,\\
\frac{1}{(k\ln k)^2}\;\;\mbox{ if }n=3,
\end{cases}
\end{split}\end{equation}
\begin{equation}\label{errFuera2}
\bigg|\sum_{j\neq 1}\int_{B(\oxi_j,\frac{\delta}{k})}E\oZ_{n+1}\bigg|\leq C\begin{cases}
\frac{1}{k^{n-1}}\;\;\mbox{ if }n\geq 4,\\
\frac{1}{(k\ln k)^2}\;\;\mbox{ if }n=3,
\end{cases}
\bigg|\sum_{j=1}^k\int_{B(\uxi_j,\frac{\delta}{k})}E\oZ_{n+1}\bigg|\leq C\begin{cases}
\frac{1}{k^{n-1}}\;\;\mbox{ if }n\geq 4,\\
\frac{1}{(k\ln k)^2}\;\;\mbox{ if }n=3.
\end{cases}
\end{equation}
Claim 1 follows using estimates \eqref{uno} and \eqref{due}.

\noindent {\bf Proof of Claim 2.} Noticing that
$$\bigg|\int_{\R^n}(\overline{\zeta}_1-1)E\oZ_{n+1}\bigg|\leq C\bigg|\int_{\{|y-\oxi_1|\geq\frac{\delta}{k}\}}E\oZ_{n+1}\bigg|,$$
the result follows using \eqref{Zfuera} and  \eqref{errFuera2}.

\noindent {\bf Proof of Claim 3.} Decomposing the non linear term as in the proof of Proposition \ref{existPhi1} and using Proposition \ref{improvedEstimates2} it can be seen that
$$\bigg|\int_{\R^n}\mathcal{N}(\ophi_1,\phi)\oZ_{n+1}\bigg|\leq Ck^{3-n-\frac{n}{q}}\int_{\R^n}U^{p-1}|Z_{n+1}|,$$
and the claim holds for $n\geq 4$. If $n=3$ it follows from estimates \eqref{ftilde1}, \eqref{ftilde2} and the fact that
$$\bigg|\int_{\R^n}\mathcal{N}(\ophi_1,\phi)\oZ_{n+1}\bigg|\leq C\|\lambda^{\frac{n+2}{2}}\mathcal{N}(\ophi_1,\phi)(\lambda y+\oxi_1)\|_{**}\left(\int_{\R^n}\frac{dy}{(1+|y|)^{2n}}\right)^{\frac{q-1}{q}}.$$

\medskip

Next we proceed to compute  $\oc_3$. We write
$$\oc_{3}(\ell, t)=\int_{\R^n}E\oZ_{3}-\int_{\R^{n}}(1-\ozeta_1)E\oZ_{3}+\gamma\int_{\R^n}\mathcal{N}(\ophi_1,\phi)\oZ_{3}=0,$$
and we affirm that, for $\ell$ and $t$ as in \eqref{par1},

\noindent {\bf Claim 4:}
$$\int_{\R^n}E\oZ_3=\begin{cases}
D_n\frac{t\ell^{\frac{n}{n-2}}}{k^{n+1-\frac{2}{n-1}}}\left[d_n\frac{\ell}{t^{n-1}}-1\right]+\frac{1}{k^{n+1}}\Theta_k(\ell,t)\;\;\mbox{ if }n\geq 4,\\
F\frac{\ell^3 t}{\sqrt{\ln k}(k\ln k)^3}\left[f\frac{\ell}{t^2}-1\right]+\frac{1}{k^3(\ln k)^4}\Theta_k(\ell,t)\;\;\mbox{ if }n=3.
\end{cases}$$

\noindent {\bf Claim 5:}
$$\int_{\R^{n}}(1-\ozeta_1)E\oZ_{3}=
\begin{cases}
\frac{1}{k^{n+1}}\Theta_k(\ell,t)\;\;\mbox{ if }n\geq 4,\\
\frac{1}{k^3(\ln k)^4}\Theta_k(\ell,t)\;\;\mbox{ if }n=3.
\end{cases}$$

\noindent {\bf Claim 6:}
$$\int_{\R^n}\mathcal{N}(\ophi_1,\phi)\oZ_{3}=
\begin{cases}
\frac{1}{k^{\alpha}}\Theta_k(\ell,t)\;\;\mbox{ if }n\geq 4,\\
\frac{1}{k^3(\ln k)^4}\Theta_k(\ell,t)\;\;\mbox{ if }n=3,
\end{cases}$$
where $\alpha>n+1-\frac{2}{n-1}$.

These claims together imply the validity of the first equation in (\ref{systLT}) and (\ref{systLTn3}).

\noindent {\bf Proof of Claim 4.} We decompose again as
\begin{equation}\label{dec3}
\int_{\mathbb{R}^n}E\oZ_{3}=\int_{B(\oxi_1,\frac{\delta}{k})}E\oZ_{3}+\int_{Ext}E\oZ_{3}+\sum_{j\neq 1}\int_{B(\oxi_j,\frac{\delta}{k})}E\oZ_{3}+\sum_{j= 1}^k\int_{B(\uxi_j,\frac{\delta}{k})}E\oZ_{3},
\end{equation}
Proceeding as in \eqref{Zfuera} and \eqref{errFuera2} we get
\begin{equation}\label{fuera31}
\bigg|\int_{Ext}E\oZ_{3}\bigg|\leq C\lambda^{n-1}k^{n-1}\lambda^{-\frac{n-2}{2}}k^{-n\frac{q-1}{q}}\|(1+|y|)^{n+2-\frac{2n}{q}}E\|_{L^q(Ext)},
\end{equation}
\begin{equation}\label{fuera32}
\bigg|\sum_{j\neq 1}\int_{B(\oxi_j,\frac{\delta}{k})}E\oZ_{3}\bigg|\leq C\sum_{j\neq 1}\frac{\lambda^{n-1}(\lambda k)^{2-\frac{n}{q}}}{|\oxi_j-\oxi_1|^{n-1}}\|\, (1+ |y|)^{{n+2} - \frac {2n} q}\, \la^{n+2 \over 2} \gamma^{-1} E (\oxi_j + \la y ) \|_{L^q( |y| < {\delta \over \la k})},
\end{equation}
\begin{equation}\label{fuera33}
\bigg|\sum_{j= 1}^k\int_{B(\uxi_j,\frac{\delta}{k})}E\oZ_{3}\bigg|\leq C\sum_{j=1}^k\frac{\lambda^{n-1}(\lambda k)^{2-\frac{n}{q}}}{|\uxi_j-\oxi_1|^{n-1}}\|\, (1+ |y|)^{{n+2} - \frac {2n} q}\, \la^{n+2 \over 2} \gamma^{-1} E (\uxi_j + \la y ) \|_{L^q( |y| < {\delta \over \la k})}.
\end{equation}
For the first integral in \eqref{dec3} we separate as in \eqref{mainError}.
Noticing that $\oU_j(\lambda y +\oxi_1)$ is even with respect to the third coordinate, it follows that
\begin{equation}\label{1}
\lambda^{\frac{n-2}{2}}\int_{B(0,\frac{\delta}{\lambda k})}U^{p-1}\oU_j(\lambda y+\oxi_1) Z_{3}=0.
\end{equation}
Furthermore, using \eqref{e0},
$$
\lambda^{\frac{n-2}{2}}\int_{B(0,\frac{\delta}{\lambda k})}U^{p-1}\uU_j(\lambda y+\oxi_1) Z_{3}=c_n\frac{\lambda^{n-2}\tau\lambda I_2}{|\oxi_1-\uxi_j|^{n}}\left(1+\lambda^2\Theta_k(\ell,t)\right),
$$
and
$$
\lambda^{\frac{n-2}{2}}\int_{B(0,\frac{\delta}{\lambda k})}U^{p-1}U(\lambda y+\oxi_1) Z_3=\tilde{c}_n\lambda^{\frac{n-2}{2}}\tau\lambda I_2\left(1+\frac{\lambda^2}{1+|\oxi_1|^2}\Theta_k(\ell,t)\right),
$$
where
$I_2:=\int_{\R^n}U^{p-1}y_3Z_3.$
One also can compute the lower order terms
$$
\bigg|\sum_{j=1}^k\lambda^{\frac{n+2}{2}}\int_{B(0,\frac{\delta}{\lambda k})}\uU_j^p(\lambda y+\oxi_1) Z_{3}\bigg|\leq C (\lambda k)^{-1}(\lambda k)^{n+2}
$$
and
\begin{equation}\label{est12}
\bigg|\lambda^{\frac{n+2}{2}}\int_{B(0,\frac{\delta}{\lambda k})}U^p(\lambda y+\oxi_1) Z_{3}\,dy\bigg|\leq C\lambda^{\frac{n}{2}}k^{-1}.
\end{equation}
Thus, decomposing as in \eqref{mainError} Claim 4 is obtained from estimates \eqref{fuera31}-\eqref{est12} together with \eqref{formufinal}, \eqref{formufinal2} and \eqref{tre}.

\noindent {\bf Proof of Claim 5.} It follows straightforward from \eqref{fuera31}, \eqref{fuera32} and \eqref{fuera33}.

\noindent {\bf Proof of Claim 6.}
Assume first $n\geq 5$ and let us decompose $\mathcal{N}(\ophi_1,\phi)=f_1+f_2+f_3+f_4$ as in the proof of Proposition \ref{existPhi1}. Changing variables and using \eqref{f1} it can be seen that
$$
\bigg|\int_{\R^n}f_1(y)\oZ_3(y)\,dy \bigg|\leq C\lambda^{\frac{n-2}{2}}\int_{B(0,\frac{\delta}{\lambda k})}U(y)^{p-2}|\oophi_1(y)||Z_3|.
$$
Notice that in this region $\lambda^{1/2}\leq \frac{c}{|y|}$ and thus, by Proposition \ref{improvedEstimates2} we can write
$$|\oophi_1(y)|\leq C\frac{\lambda^{\beta}\lambda^{\frac{n-2}{2}-\beta}}{(1+|y|)^2}\leq C\frac{\lambda^{\beta}}{(1+|y|)^{n-2\beta}},$$
where $\beta:=\frac{3}{2}-\frac{1}{n-1}+\varepsilon$, $\varepsilon>0$ small. Replacing above we get
$$
\bigg|\int_{\R^n}f_1(y)\oZ_3(y)\,dy\bigg| \leq \frac{C}{k^{2\beta+n-2}}\int_{\R^n}\frac{1}{(1+|y|)^{-n+6}}\frac{|y_3|}{(1+|y|)^n}\frac{1}{(1+|y|)^{n-2\beta}}\,dy \leq \frac{C}{k^{2\beta+n-2}}
$$
for $\varepsilon$ small enough. Notice also that $2\beta+n-2>n+1-\frac{2}{n-1}$, that is the order of the main term. Likewise, using the estimate on $\oophi_1$,
$$
\bigg|\int_{\R^n}f_2(y)\oZ_3(y)\,dy\bigg| \leq C\lambda^{\frac{n-2}{2}}\int_{\{|y|>c\lambda^{-1/2}\}}\frac{U^{p-1}|Z_3|}{(1+|y|)^2}\,dy\leq C\lambda^{\frac{n-2}{2}+2}\int_{\R^n}\frac{|y_3|}{(1+|y|)^{n+2}}\,dy\leq \frac{C}{k^{n+2}}.
$$
To estimate the projection of $f_3$ we first point out that
$f_3 \approx \ozeta_1 p|\oU_1|^{p-1}\psi$. Due to the cancellation in \eqref{1} the main order term in Claim 4 is rather small, and this makes necessary sharp estimates on the size of the projection of the nonlinear term. Indeed, to prove this claim we will have to make use of the decomposition of $\psi$ in a large but symmetric part (that happens to be orthogonal to $\oZ_3$) and a non symmetric but small part specified in Proposition \ref{symEstPsi}. Thus,
$$\begin{aligned}
\int_{\R^n}f_3(y)\oZ_3(y)\,dy&\approx \lambda^{\frac{n-2}{2}}\int_{B(0,\frac{\delta}{\lambda k})}U(y)^{p-1}\psi(\oxi_1+\lambda y)Z_3\\
&=\lambda^{\frac{n-2}{2}}\int_{B(0,\frac{\delta}{\lambda k})}U^{p-1}\psi^s(y)Z_3\,dy+\lambda^{\frac{n-2}{2}}\int_{B(0,\frac{\delta}{\lambda k})}U^{p-1}\psi^*(y)Z_3\,dy.
\end{aligned}$$
The first integral in the right hand side vanishes due to the oddness of $Z_3$ and the evennes of $\psi^s$ in the third coordinate. The second can be estimated as
\begin{equation*}\begin{split}
\bigg|\lambda^{\frac{n-2}{2}}\int_{B(0,\frac{\delta}{\lambda k})}U^{p-1}\psi^*(y)Z_3\,dy\bigg|&\leq C\lambda^{\frac{n-2}{2}+1}(\|\oophi_1\|_*+\|\psi\|_*+\tau o_k(1))\int_{\R^n}U^{p-1}|y|(1+|y|)|Z_3|\,dy\\
&\leq \frac{C}{k^{n-1+\frac{n}{q}}},
\end{split}\end{equation*}
and hence, choosing $\frac{n}{2}<q<\frac{n}{2-\frac{2}{n-1}}$
we conclude that
$$\bigg|\int_{\R^n}f_3(y)\oZ_3(y)\,dy\bigg|\leq \frac{C}{k^\alpha}\quad \mbox{ with }\alpha>n+1-\frac{2}{n-1}.$$
Analogously, noticing that $f_4\approx |\oU_1|^{p-2}\left(\sum_{j=1}^k(\ophi_j+\uphi_j)+\psi\right)^2$, estimate
$$\bigg|\int_{\R^n}f_4(y)\oZ_3(y)\,dy\bigg|\leq \frac{C}{k^\alpha}\quad \mbox{ with }\alpha>n+1-\frac{2}{n-1},$$
follows by Proposition \ref{improvedEstimates}. This completes the proof of Claim 6 for $n\geq 5$.

Consider now the cases $n=3,4$. Using the decomposition of $\oophi_1$ found in Proposition \ref{estSimPart} we notice that 
$$\int_{\R^n}\tilde{f}_1^s(y)Z_3(y)\,dy =0,\qquad \int_{\R^n}\tilde{f}_2^s(y)Z_3(y)\,dy =0,$$
and we obtain, for $n=3$,
\begin{equation*}\begin{split}
&\bigg|\int_{\R^n}f_1(y)\oZ_3(y)\,dy\bigg|=\bigg|\int_{\R^n}\tilde{f}^*_1(y)Z_3(y)\,dy\bigg|\\
&\leq  C\left(\lambda^{\frac{1}{2}}\int_{B(0,\frac{\delta}{\lambda k})}U(y)^{p-2}|\overline{\overline{\phi_1^*}}(y)||Z_3(y)|\,dy+\frac{\lambda^{\frac{1}{2}}}{k(\ln k)^3}\int_{B(0,\frac{\delta}{\lambda k})}|y|U(y)^{p-2}|\oophi_1(y)||Z_3(y)|\,dy\right)\\
&\leq \frac{C}{k^3(\ln k)^5},
\end{split}\end{equation*}
and
\begin{equation*}\begin{split}
\bigg|\int_{\R^n}f_2(y)\oZ_3(y)\,dy\bigg|&=\bigg|\int_{\{|y|>c\lambda^{-1/2}\}}f_2^*(y)Z_3(y)\,dy\bigg|=\bigg|\int_{\{|y|>c\lambda^{-1/2}\}}U(y)^{p-1}\overline{\overline{\phi_1^*}}(y)Z_3(y)\,dy\bigg|\\
&\leq C\frac{\lambda^\frac{1}{2}\lambda}{k(\ln k)^3}\int_{\R^n}\frac{1}{(1+|y|)^2}\frac{|y_3|}{(1+|y|)^n}\leq\frac{C}{k^3(\ln k)^5},
\end{split}\end{equation*}
where in the last inequalities we have applied Proposition \ref{improvedEstimates2} and Proposition \ref{estSimPart}. 

Likewise, for $n=4$,
$$\bigg|\int_{\R^n}f_1(y)\oZ_3(y)\,dy\bigg|\leq \frac{C}{k^5},\qquad \bigg|\int_{\R^n}f_2(y)\oZ_3(y)\,dy\bigg|\leq \frac{C}{k^5}.$$
The terms involving $\tilde{f}_3$ and $\tilde{f}_4$ are estimated in a similar way and Claim 6 follows.

\section{Remark \ref{r16}: The general construction} \label{general}

\noindent This section is devoted to the constructions described in Remark \ref{r16}.
The first is the construction of the {\it doubling of the equatorial} $\Gamma$ with an even number of circles, which is done in subsection \ref{even}.
The second is a combination of the {\it doubling} and the {\it desingularization of the equatorial}
with an odd number of circles. This is done in subsection \ref{odd}.

\subsection{Even number of circles.} \label{even}
Let $m $ be a fixed integer. Let $\tau_i \in (0,1)$, $i=1, \ldots , m,$ and fix the  points
$$
 \overline P_i := (\sqrt{1-\tau_i^2 } , 0 , \tau_i , 0 , \ldots , 0) , \quad \underline P_i :=  (\sqrt{1-\tau_i^2 } , 0 , -\tau_i , 0 , \ldots , 0).
$$
Let $ \la_i \in (0,1)$, $i=1, \ldots , m,$ be  positive numbers, and define $R_i$ as
$
\la_i^2 + R_i^2 = 1.
$
We use the notation
$$
\la = (\la_1 , \ldots , \la_m) , \quad \tau = (\tau_1 , \ldots , \tau_m).
$$
Let $k$ be an integer number and
\begin{equation}\label{a11}
\textbf{u}_{2m} [\la, \tau ] (y) := U(y) -\sum_{j=1}^k \bigg[ \sum_{i=1}^m \underbrace{ \la^{-{n-2 \over 2}}  U\left( {y-\oxi_{ij} \over \la} \right)}_{\oU_{ij} (y)}
+ \sum_{i=1}^m \underbrace{ \la^{-{n-2 \over 2}}  U\left( {y- \uxi_{ij} \over \la} \right)}_{ \uU_{ij} (y) } \bigg]
\end{equation}
for $y \in \R^n$, 
where, for $i=1, \ldots , m$, and $j=1, \ldots , k$,
$$\begin{aligned}
\oxi_{ij} &:= R_i (\sqrt{1-\tau_i^2} \cos \theta_j , \sqrt{1-\tau_i^2} \sin \theta_j , \tau_i , 0, \ldots , 0 ), \\
 \uxi_{ij} &:= R_i (\sqrt{1-\tau_i^2} \cos \theta_j , \sqrt{1-\tau_i^2} \sin \theta_j , -\tau_i , 0, \ldots , 0 ),
 \quad {\mbox {with}} \quad \theta_j := 2\pi {j-1 \over k}.
\end{aligned}$$
Observe that the function \eqref{a11} satisfies the symmetries \eqref{ikt}, \eqref{ip} and \eqref{ir}.
We assume that the integer $k$ is large, and that the parameters  $\la $ and $\tau$ are given by
\begin{equation}\begin{split} \label{par1nn}
\la_i &:= {\ell_i^{2\over n-2} \over k^2} ,  \quad  \tau_i := {t_i \over k^{1-{2\over n-1}} },  \quad {\mbox {if}} \quad n\geq 4, \\
\la_i &:= {\ell_i^{2} \over k^2 (\ln k)^2 } , \quad  \tau_i := {t_i \over \sqrt{\ln k} },  \quad {\mbox {if}} \quad n=3,
\end{split}  \quad
\quad {\mbox {
where}} \quad
\eta <  \ell_i , \, t_i < \eta^{-1}
\end{equation}
for some $\eta$ small and fixed, independent of $k$, for any $k$ large enough. The {\it doubling of the equatorial} $\Gamma$ with an even number of circles is the content of next

\begin{theorem}\label{theo1}
Let $n\geq 3$ and let $k$ be a positive integer. Then for any sufficiently large $k$ there is a finite energy solution to \eqref{prob} of the form
$$u(y)=\textbf{u}_{2m} [\la, \tau ] (y)  +o_k(1)(1+|\la|^{-{n-2 \over 2}}),$$
where
the term $o_k(1)\rightarrow 0$ uniformly on compact sets of $\R^n$ as  $k \to \infty$.
\end{theorem}
The solution in Theorem \ref{theo1} has the form
$$u (y) = \textbf{u}_{2m} [\la, \tau ] (y) + \phi (y ), \quad \phi=  \psi + \sum_{j=1}^k \sum_{i=1}^m ( \bar \phi_{ij} + \underline \phi_{ij} )$$
where $\overline{\phi}_{ij}$, $\underline{\phi}_{ij}$, $i=1, \ldots , m$, $j=1, \ldots , k$, and $\psi$ solve the following system of coupled non linear equations
\begin{equation}\label{system1nn}
\Delta \ophi_{ij}+p\gamma |\textbf{u}_{2m}|^{p-1}\ozeta_{ij}\ophi_{ij}+\ozeta_{ij}\left[p\gamma |\textbf{u}_{2m}|^{p-1}\psi+E_{2m}+\gamma N(\phi)\right]=0, \; \;
\end{equation}
\begin{equation}\label{system2nn}
\Delta \uphi_{ij}+p\gamma |\textbf{u}_{2m}|^{p-1}\uzeta_{ij}\uphi_{ij}+\uzeta_{ij}\left[p\gamma |\textbf{u}_{2m}|^{p-1}\psi+E_{2m}+\gamma N(\phi)\right]=0, \; \;
\end{equation}
\begin{equation}\label{system3nn}\begin{split}
\Delta\psi&+p\gamma U^{p-1}\psi+\left[p\gamma(|\textbf{u}_{2m}|^{p-1}-U^{p-1})(1-\sum_{j=1}^k \sum_{i=1}^m(\ozeta_{ij}+\uzeta_{ij}))\right.\\
&\left.+p\gamma U^{p-1}[\sum_{j=1}^k \sum_{i=1}^m(\ozeta_{ij}+\uzeta_{ij})\right]\psi+ p\gamma |\textbf{u}_{2m}|^{p-1}\sum_{j=1}^k \sum_{i=1}^m (1-\ozeta_{ij})\ophi_{ij}\\
&+p\gamma |\textbf{u}_{2m}|^{p-1}\sum_{j=1}^k \sum_{i=1}^m (1-\uzeta_{ij})\uphi_{ij}+\left(1-\sum_{j=1}^k \sum_{i=1}^m(\ozeta_{ij}+\uzeta_{ij})\right)(E_{2m}+\gamma N(\phi))=0.
\end{split}\end{equation}
The functions $\overline \zeta_{ij}$ are defined as $\ozeta_j$ in \eqref{ccut} with $\bar \xi_j$ replaced by  $\bar \xi_{ij}$,
and 
$$\underline \zeta_{ij}(y):=\overline{\zeta}_{ij}(y_1,y_2,-y_3,\ldots,y_n),\quad 
E_{2m} (y) := \Delta \textbf{u}_{2m} + \gamma |\textbf{u}_{2m}|^{p-1} \textbf{u}_{2m}, \quad y \in \R^n,
$$
and $
N(\phi):=|\textbf{u}_{2m}+\phi|^{p-1}(\textbf{u}_{2m}+\phi)-|\textbf{u}_{2m}|^{p-1}\textbf{u}_{2m}-p|\textbf{u}_{2m}|^{p-1}\phi $.
One can prove that
$$
\| E_{2m} \|_{**} \leq C k^{1-\frac{n}{q}}\;\;\mbox{if }n\geq 4,\quad
\| E_{2m} \|_{**} \leq C(\ln k)^{-1}\;\;\mbox{if }n=3.
$$
Denoting $\hat{y}:=(y_1,y_2)$ and $y':=(y_3,\ldots,y_n)$,
we assume that the functions
$\ophi_{ij}$ and $\uphi_{ij}$ satisfy
$$
\ophi_{ij}(\hat{y},y')=\ophi_{i1}(e^{2\pi\frac{ \textcolor{black}{(j-1)}}{k}i} \hat{y},y'),\quad
\quad
\ophi_{i1}(y)=|y|^{2-n}\ophi_{i1}(|y|^{-2}y),
$$
$$
\ophi_{i1}(y_1,\ldots,y_\alpha,\ldots,y_n)=\ophi_{i1}(y_1,\ldots,-y_\alpha,\ldots,y_n),\;\; \alpha=2,4,\ldots,n,
$$
and
$$
\uphi_{ij}(y)=\ophi_{ij}(y_1,y_2,-y_3,\ldots,y_n).
$$
Moreover
$$(\ophi_{ij}+\uphi_{ij})(y)=(\ophi_{ij}+\uphi_{ij})(y_1,y_2,-y_3,\ldots,y_n),$$
as well as $(\ozeta_{ij}+\uzeta_{ij})$ and $(\ozeta_{ij}\ophi_{ij}+\uzeta_{ij}\uphi_{ij})$.

\noindent For $\rho>0$ small and fixed we assume in addition
$$
\sum_{i=1}^m \|\overline{\ophi}_{i1}\|_*\leq \rho,
$$
where $\overline{\ophi}_{i1}(y):=\lambda^{\frac{n-2}{2}}\overline{\phi}_{i1}(\oxi_{i1}+\lambda y)$ and $\|\cdot\|_*$ is defined in \eqref{normStar}.

Arguing as in Proposition \ref{existPsi}, one proves that there exists a unique solution $\psi=\Psi(\overline{\ophi}_{11} , \ldots , \overline{\ophi}_{m1})$ of \eqref{system3nn}, satisfying \eqref{sym1}, \eqref{sym2} and \eqref{sym3}. Besides
\begin{equation*}\begin{split}
\|\psi\|_*&\leq C\left(\sum_{i=1}^m \|\overline{\ophi}_{i1}\|_*+k^{1-\frac{n}{q}}\right)\;\;\mbox{if }n\geq 4,\quad
\|\psi\|_*\leq C\left( \sum_{i=1}^m \|\overline{\ophi}_{i1}\|_*+(\ln k)^{-1}\right)\;\;\mbox{if }n=3.
\end{split}\end{equation*}
We replace the solution $\psi=\Psi(\overline{\ophi}_{11} , \ldots , \overline{\ophi}_{m1})$ of \eqref{system3nn} in \eqref{system1nn} and \eqref{system2nn}. Using the symmetries we described before, it is enough to solve \eqref{system1nn} for $j=1$. We are thus left with a system of $m$
equations in $\overline \phi_1 = (\overline \phi_{11} , \ldots , \overline \phi_{m1} )$  unknowns
$$
\Delta \ophi_{i1}+p\gamma |\textbf{u}_{2m}|^{p-1}\ozeta_{i1}\ophi_{i1}+\ozeta_{i1}\left[p\gamma |\textbf{u}_{2m}|^{p-1}\psi+E_{2m}+\gamma N(\phi)\right]=0, \quad i=1, \ldots , m.
$$
Instead of solving it directly, we first solve the auxiliary problem
\begin{equation}\label{phiEven}
\Delta \ophi_{i1}+p\gamma |\overline{U}_{i1}|^{p-1}\ophi_{i1}+\ozeta_{i1} E_{2m} +\gamma\mathcal{N}_i(\overline{\phi}_1,\phi)=c_{i3}\oU_{i1}^{p-1}\oZ_{i3}+c_{i,n+1}\oU_{i1}^{p-1}\oZ_{i,n+1},\end{equation}
where
$$\mathcal{N}_i(\overline{\phi}_1,\phi):=p(|\textbf{u}_{2m}|^{p-1}\ozeta_1-|\overline{U}_{i1}|^{p-1})\ophi_{i1}+\ozeta_1\left[p|\textbf{u}_{2m}|^{p-1}\Psi(\overline{\ophi}_1)
+N(\phi)\right],$$
$$\overline{Z}_{i\alpha}(y):=\lambda^{-\frac{n-2}{2}}Z_\alpha\left(\frac{y-\oxi_{i1}}{\lambda}\right),\;\;\alpha=3, n+1$$
and
$$
c_{i3}:=\frac{\int_{\R^n}(\ozeta_{i1} E_{2m}+\gamma\mathcal{N}_i(\overline{\phi}_1,\phi))\oZ_{i3}}{\int_{\R^n}\oU_{i1}^{p-1}\oZ^2_{i3}},\qquad c_{i, n+1}:=\frac{\int_{\R^n}(\ozeta_{i1} E_{2m} +\gamma\mathcal{N}_i (\overline{\phi}_1,\phi))\oZ_{i, n+1}}{\int_{\R^n}\oU_{i1}^{p-1}\oZ^2_{i,n+1}}.
$$
Arguing as in Proposition \ref{existPhi1}, one proves that
there exists a unique solution $\overline{\phi}_{i1}=\overline{\phi}_{i1}(l,t)$ of \eqref{phiEven}, that satisfies
$$\|\oophi_{i1}\|_*\leq C k^{-\frac{n}{q}}\mbox{ if }n\geq 4,\qquad \|\oophi_{i1}\|_*\leq \frac{C}{k\ln k}\mbox{ if }n=3,$$
and
$$\|\overline{\overline{\mathcal{N}}}(\overline{\phi}_{i1},\phi)\|_{**}\leq Ck^{-\frac{2n}{q}}\mbox{ if }n\geq 4,\qquad \|\overline{\overline{\mathcal{N}}}(\overline{\phi}_{i1},\phi)\|_{**}\leq \frac{C}{(k\ln k)^2}\mbox{ if }n=3,$$
where $\oophi_{i1}(y):=\lambda^\frac{n-2}{2}\ophi_{i1}(\oxi_{i1}+\lambda y)$ and $\overline{\overline{\mathcal{N}}}(\overline{\phi}_{i1},\phi)(y):=\lambda^\frac{n+2}{2}\mathcal{N}(\overline{\phi}_{i1},\phi)(\oxi_{i1}+\lambda y)$. Furthermore, proceeding as in Proposition \ref{estSimPart} and Proposition \ref{symEstPsi}, there exist  decompositions
$$\oophi_{i1}=\overline{\overline{\phi_{i1}^s}}+\overline{\overline{\phi_{i1}^*}},\qquad \psi(\oxi_{i1}+\lambda y)=\psi^s(y)+\psi^*(y),$$
where $\overline{\overline{\phi_{i1}^s}}$ and $\psi^s$ are even with respect to $y_3$
and
$$|\overline{\overline{\phi_{i1}^*}}(y)|\leq \begin{cases}
C\frac{\lambda^{\frac{n-2}{2}}}{k}\frac{1}{1+|y|}\mbox{ if }n\geq 4,\\
C\frac{\lambda^{1/2}}{k(\ln k)^3}\frac{1}{1+|y|^\alpha},\quad 0<\alpha<1,\quad\mbox{ if }n\geq 3,
\end{cases}\qquad |\psi^*(y)|\leq C\lambda k^{1-\frac{n}{q}}|y|(1+|y|).$$

In order to complete the proof of Theorem \ref{theo1}, we need to find
 positive parameters $\ell_1 , \ldots \ell_m$ and $t_1 , \ldots , t_m$   entering in the definition of $\lambda$ and $\tau$ in \eqref{par1nn}
 so that for all $i=1, \ldots , m$
 \begin{equation}\label{sisnn}
 c_{i3} (\ell, t) = c_{i, n+1} (\ell , t) = 0, \quad \ell = (\ell_1 , \ldots \ell_m), \quad t= ( t_1 , \ldots , t_m ).
 \end{equation}
 In dimension $n\geq 4$, this system decouples and becomes
$$\begin{aligned}
&a_n \ell_i^2-1+\frac{1}{k^{2 {n-3 \over n-1}}}\Theta_{i, n+1, k}(\ell_1 , \ldots \ell_m ,t_1 , \ldots , t_m ) =0, \quad i=1, \ldots , m, \\
&b_n{ \ell_i \over t_i^{n-1} }-1+\frac{1}{k^{n-3 \over n-1} }\Theta_{i, 3, k}(\ell_1 , \ldots \ell_m ,t_1 , \ldots , t_m)=0, \quad i=1, \ldots , m,
\end{aligned}$$
where $a_n$, $b_n$ are positive constants that are independent of $k$, and $\Theta_{i, n+1, k}$, $\Theta_{i, 3, k}$ are smooth functions of their argument, which are uniformly bounded, together with their first derivatives, as $k \to \infty$.

In dimension $n=3$, the system becomes
$$\begin{aligned}
&a_3 \ell_i^2-1+\frac{\ln \ln k }{\ln k }\Theta_{i, 4, k}(\ell_1 , \ldots \ell_m ,t_1 , \ldots , t_m ) =0, \quad i=1, \ldots , m, \\
&b_3{ \ell_i \over t_i^{2} }-1+\frac{\ln \ln k }{\ln k } \Theta_{i, 3, k}(\ell_1 , \ldots \ell_m ,t_1 , \ldots , t_m)=0, \quad i=1, \ldots , m,
\end{aligned}$$
where $a_3$, $b_3$ are positive constants, and $\Theta_{i, 4, k}$, $\Theta_{i, 3, k}$ are smooth functions of their argument, which are uniformly bounded, together with their first derivatives, as $k \to \infty$.
A fixed point argument gives the existence of $\ell$ and $t$ solutions to \eqref{sisnn}.
This concludes the proof of Theorem \ref{theo1}.

\subsection{Odd number of circles.} \label{odd}
Let $ \mu \in (0,1)$ and define $R $ so that
$
\mu^2 + R^2 = 1.
$
Let $k$, $m$ be  integer numbers and
\begin{equation}\label{a11n}
\textbf{u}_{2m+1} [\mu, \la, \tau ] (y) := U(y) -\sum_{j=1}^k \left[ U_j (y) + \sum_{i=1}^m \left( \oU_{ij} (y)
+  \uU_{ij} (y) \right) \right]
\end{equation}
where $\la$, $\tau$, $\oU_{ij}$ and $\uU_{ij}$ are defined at the beginning of subsection \ref{even} and \eqref{a11}, while
$$
U_j (y) := \mu^{-{n-2 \over 2}} U({y-\xi_j \over \mu} ), \quad \xi_j:= R ( \cos \theta_j ,  \sin \theta_j , 0 , 0, \ldots , 0 ) \quad \theta_j := 2\pi {j-1 \over k}.
$$
The function \eqref{a11n} satisfies the symmetries \eqref{ikt}, \eqref{ip} and \eqref{ir}.
We assume that the integer $k$ is large, and that the parameters  $\mu$, $\la $ and $\tau$ are given by
\begin{equation}\begin{split} \label{par1nnn}
\mu &:= {\ell^{2\over n-2} \over k^2}    \quad {\mbox {if }}  n\geq 4, \quad
\mu := {\ell^{2} \over k^2 (\ln k)^2 }    \quad {\mbox {if }}  n=3,\\
\la_i &:= {\ell_i^{2\over n-2} \over k^2} ,  \quad  \tau_i := {t_i \over k^{1-{2\over n-1}} },  \quad {\mbox {if }}  n\geq 4, \quad 
\la_i := {\ell_i^{2} \over k^2 (\ln k)^2 } , \quad  \tau_i := {t_i \over \sqrt{\ln k} },  \quad {\mbox {if }}  n=3,
\end{split}
\end{equation}
where $\eta <  \ell , \ell_i \, , \, t_i < \eta^{-1}$ for some $\eta$ small and fixed, independent of $k$, for any $k$ large enough. We have
\begin{theorem}\label{theo2}
Let $n\geq 3$ and let $k$ be a positive integer. Then for any sufficiently large $k$ there is a finite energy solution to \eqref{prob} of the form
$$u(y)=\textbf{u}_{2m+1} [\mu, \la, \tau ] (y)  +o_k(1)(1+(\mu + |\la|)^{-{n-2 \over 2}}),$$
where
the term $o_k(1)\rightarrow 0$ uniformly on compact sets of $\R^n$ as  $k \to \infty$.
\end{theorem}
The solution in Theorem \ref{theo2} has the form
$$
u (y) = \textbf{u}_{2m+1} [\mu, \la, \tau ] (y) + \phi (y ), \quad \phi=  \psi + \sum_{j=1}^k [  \phi_j + \sum_{i=1}^m ( \bar \phi_{ij} + \underline \phi_{ij} )],
$$
with
\begin{equation}\label{casa}
\phi_j (\hat{y},y')=\phi_{1}(e^{2\pi\frac{ \textcolor{black}{(j-1)}}{k}i} \hat{y},y'), \quad \ophi_{ij}(\hat{y},y')=\ophi_{i1}(e^{2\pi\frac{ \textcolor{black}{(j-1)}}{k}i} \hat{y},y'),
\end{equation}
where $\hat{y}:=(y_1,y_2)$ and $y':=(y_3,\ldots,y_n)$. The functions $ \phi_j$,
$\ophi_{ij}$ and $\uphi_{ij}$ also satisfy
$$
\phi_{1}(y)=|y|^{2-n}\phi_{1}(|y|^{-2}y), \quad 
\ophi_{i1}(y)=|y|^{2-n}\ophi_{i1}(|y|^{-2}y),
$$
$$
\ophi_{i1}(y_1,\ldots,y_\alpha,\ldots,y_n)=\ophi_{i1}(y_1,\ldots,-y_\alpha,\ldots,y_n),\;\; \alpha=2,4,\ldots,n,
$$
and
$$
\uphi_{ij}(y)=\ophi_{ij}(y_1,y_2,-y_3,\ldots,y_n).
$$
Moreover
$$
\phi_j (y) =  \phi_j (y_1,y_2,-y_3,\ldots,y_n), \quad
(\ophi_{ij}+\uphi_{ij})(y)=(\ophi_{ij}+\uphi_{ij})(y_1,y_2,-y_3,\ldots,y_n).$$
Thanks to \eqref{casa}, it is enough to describe $ \phi_1$, $\overline{\phi}_{i1}$,  $i=1, \ldots , m$.  The functions $ \phi_1$, $\overline{\phi}_{i1}$ and $\psi$ solve the following system of coupled non linear equations
$$\begin{aligned}
\Delta\psi&+p\gamma U^{p-1}\psi+\left[p\gamma(|\textbf{u}_{2m+1}|^{p-1}-U^{p-1})\left(1-\sum_{j=1}^k \left[\zeta_j +\sum_{i=1}^m(\ozeta_{ij}+\uzeta_{ij})\right]\right)\right.\\
&\left.+p\gamma U^{p-1}\left[\sum_{j=1}^k [ \zeta_j + \sum_{i=1}^m(\ozeta_{ij}+\uzeta_{ij})\right]\right]\psi+ p\gamma |\textbf{u}_{2m+1}|^{p-1}\sum_{j=1}^k \sum_{i=1}^m (1-\ozeta_{ij})\ophi_{ij}\\
&+p\gamma |\textbf{u}_{2m+1}|^{p-1}\sum_{j=1}^k \sum_{i=1}^m (1-\uzeta_{ij})\uphi_{ij}
+p\gamma |\textbf{u}_{2m+1}|^{p-1}\sum_{j=1}^k  (1-\zeta_{j})\phi_{j} \\
&+\left(1-\sum_{j=1}^k \left[\zeta_j +\sum_{i=1}^m(\ozeta_{ij}+\uzeta_{ij})\right]\right)(E_{2m +1}+\gamma N(\phi))=0,
\end{aligned}$$
$$
\Delta \phi_{1}+p\gamma |{U}_{1}|^{p-1}\phi_{1}+\zeta_{1} E_{2m +1} +\gamma\mathcal{N}(\phi)=c_{n+1}U_{1}^{p-1}Z_{n+1},
$$
and, for $i=1, \ldots , m$,
$$
\Delta \ophi_{i1}+p\gamma |\overline{U}_{i1}|^{p-1}\ophi_{i1}+\ozeta_{i1} E_{2m+1} +\gamma\mathcal{N}_i(\phi)=c_{i3}\oU_{i1}^{p-1}\oZ_{i3}+c_{i,n+1}\oU_{i1}^{p-1}\oZ_{i,n+1}.
$$
Here
$$
E_{2m+1} (y) := \Delta \textbf{u}_{2m+1} + \gamma |\textbf{u}_{2m+1}|^{p-1} \textbf{u}_{2m+1}, \quad y \in \R^n,
$$
and
$
N(\phi):=|\textbf{u}_{2m+1}+\phi|^{p-1}(\textbf{u}_{2m+1}+\phi)-|\textbf{u}_{2m+1}|^{p-1}\textbf{u}_{2m+1}-p|\textbf{u}_{2m+1}|^{p-1}\phi $. For any $j$, 
$\zeta_j$,
$\overline \zeta_{ij}$ are defined as $\zeta_j$ in \eqref{ccut} with $\bar \xi_j$ replaced respectively by $\xi_j$ and $\bar \xi_{ij}$,
and $\underline \zeta_{ij}(y):=\overline{\zeta}_{ij}(y_1,y_2,-y_3,\ldots,y_n)$.
Moreover,
$$\mathcal{N}(\phi):=p(|\textbf{u}_{2m+1}|^{p-1}\zeta_1-|\overline{U}_{1}|^{p-1})\phi_{1}+\zeta_1\left[p|\textbf{u}_{2m+1}|^{p-1}\psi
+N(\phi)\right],$$
$$\mathcal{N}_i(\phi):=p(|\textbf{u}_{2m+1 }|^{p-1}\ozeta_{i1}-|\overline{U}_{i1}|^{p-1})\ophi_{i1}+\ozeta_{i1}\left[p|\textbf{u}_{2m+1}|^{p-1}\psi
+N(\phi)\right],$$
$$
 c_{ n+1}:=\frac{\int_{\R^n}(\zeta_{1} E_{2m+1} +\gamma\mathcal{N} (\phi))Z_{ n+1}}{\int_{\R^n}U_{1}^{p-1}Z^2_{n+1}},
$$
and
$$
c_{i3}:=\frac{\int_{\R^n}(\ozeta_{i1} E_{2m+1}+\gamma\mathcal{N}_i(\phi))\oZ_{i3}}{\int_{\R^n}\oU_{i1}^{p-1}\oZ^2_{i3}},\qquad c_{i, n+1}:=\frac{\int_{\R^n}(\ozeta_{i1} E_{2m+1} +\gamma\mathcal{N}_i (\phi))\oZ_{i, n+1}}{\int_{\R^n}\oU_{i1}^{p-1}\oZ^2_{i,n+1}}.
$$
where
$$ Z_\alpha (y) := \mu^{-{n-2 \over 2}} Z_\alpha  \left(\frac{y-\xi_{1}}{\mu}\right), \quad \overline{Z}_{i\alpha}(y):=\lambda^{-\frac{n-2}{2}}Z_\alpha\left(\frac{y-\oxi_{i1}}{\lambda}\right),\;\;\alpha=3, n+1 .$$
It can be proved that
$$\begin{aligned}
\|\psi\|_*&\leq C\left(\sum_{i=1}^m \|\overline{\ophi}_{i1}\|_*+k^{1-\frac{n}{q}}\right)\;\;\mbox{if }n\geq 4,\quad
\|\psi\|_*\leq C\left( \sum_{i=1}^m \|\overline{\ophi}_{i1}\|_*+(\ln k)^{-1}\right)\;\;\mbox{if }n=3,
\end{aligned}
$$
$$\|\oophi_1\|_*\leq C k^{-\frac{n}{q}}\mbox{ if }n\geq 4,\qquad \|\oophi_1\|_*\leq \frac{C}{k\ln k}\mbox{ if }n=3,$$
$$\|\oophi_{i1}\|_*\leq C k^{-\frac{n}{q}}\mbox{ if }n\geq 4,\qquad \|\oophi_{i1}\|_*\leq \frac{C}{k\ln k}\mbox{ if }n=3,$$
where $\oophi_{1}(y):=\mu^\frac{n-2}{2}\phi_{1}(\xi_{1}+\mu y)$ and $\oophi_{i1}(y):=\lambda^\frac{n-2}{2}\ophi_{i1}(\oxi_{i1}+\lambda y)$, and the corresponding estimates on their non symmetric part.

In order to complete the proof of Theorem \ref{theo2}, we need to find
 positive parameters $\ell, \ell_1 , \ldots \ell_m$ and $t_1 , \ldots , t_m$   entering in the definition of $\mu$, $\lambda$ and $\tau$ in \eqref{par1nnn}
 so that for all $i=1, \ldots , m$
 \begin{equation}\label{sisnnn}
  c_{n+1} (\bar \ell, t)  = c_{i3} (\bar \ell, t) = c_{i, n+1} (\bar \ell , t) = 0, \quad \bar \ell = (\ell, \ell_1 , \ldots \ell_m), \quad t= ( t_1 , \ldots , t_m ).
 \end{equation}
 In dimension $n\geq 4$, this system decouples at main order and becomes
 $$\begin{aligned}
 &a_n \ell^2-1+\frac{1}{k^{2 {n-3 \over n-1}}}\Theta_{n+1, k}(\ell, \ell_1 , \ldots \ell_m ,t_1 , \ldots , t_m ) =0, \\
&a_n \ell_i^2 -1+\frac{1}{k^{2 {n-3 \over n-1}}}\Theta_{i, n+1, k}(\ell, \ell_1 , \ldots \ell_m ,t_1 , \ldots , t_m ) =0, \quad i=1, \ldots , m, \\
&b_n{ \ell_i \over t_i^{n-1} }-1+\frac{1}{k^{n-3 \over n-1} }\Theta_{i, 3, k}(\ell, \ell_1 , \ldots \ell_m ,t_1 , \ldots , t_m)=0, \quad i=1, \ldots , m,
\end{aligned}$$
where $a_n$, $b_n$ are positive constants, and $\Theta_{n+1, k}$, $\Theta_{i, n+1, k}$, $\Theta_{i, 3, k}$ are smooth functions of their argument, which are uniformly bounded, together with their first derivatives, as $k \to \infty$.

In dimension $n=3$, system \eqref{sisnnn} decouples and becomes
 $$\begin{aligned}
 &a_3 \ell^2-1+\frac{\ln \ln k }{\ln k }\Theta_{4, k}(\ell, \ell_1 , \ldots \ell_m ,t_1 , \ldots , t_m ) =0,  \\
&a_3 \ell_i^2-1+\frac{\ln \ln k }{\ln k }\Theta_{i, 4, k}(\ell, \ell_1 , \ldots \ell_m ,t_1 , \ldots , t_m ) =0, \quad i=1, \ldots , m, \\
&b_3{ \ell_i \over t_i^{2} }-1+\frac{\ln \ln k }{\ln k } \Theta_{i, 3, k}(\ell, \ell_1 , \ldots \ell_m ,t_1 , \ldots , t_m)=0, \quad i=1, \ldots , m,
\end{aligned}$$
where $a_3$, $b_3$ are positive constants, and $\Theta_{4, k}$, $\Theta_{i,4,k} $ $\Theta_{3, k}$ are smooth functions of their argument, which are uniformly bounded, together with their first derivatives, as $k \to \infty$.
A fixed point argument gives the existence of $\ell$ and $t$ solutions to \eqref{sisnnn}.

\section*{Appendix A: Some useful computations}

\noindent
{\bf Proof of \eqref{uno} and \eqref{due}.} By definition
$$
\sum_{j=2}^k {1\over |\oxi_1 - \oxi_j|^{n-2}} = {1\over R^{n-2} (1-\tau^2)^{n-2 \over 2}} \, \left(\sum_{j=2}^k {1\over [2 \, (1-\cos \theta_j)]^{n-2 \over 2} } \right) , \quad \theta_j = 2\pi {j-1 \over k}.
$$
Using the symmetry of the construction we have 
\begin{equation*}\begin{split}
\sum_{j=2}^k {1\over [2(1-\cos \theta_j)]^{n-2 \over 2} }&= \begin{cases}
\displaystyle 2\sum_{j=2}^{\frac{k}{2}} {1\over [2(1-\cos \theta_j)]^{n-2 \over 2}} +\frac{1}{4^{\frac{n-2}{2}}}\quad \mbox{ if $k$ even},\\
 \displaystyle2\sum_{j=2}^{\frac{k-1}{2}} {1\over [2(1-\cos \theta_j)]^{n-2 \over 2}}\quad \mbox{ if $k$ odd},
 \end{cases}
\end{split}
\end{equation*}
where $0<\theta_j<\pi$. Hence,
$$2(1-\cos\theta_j)=\theta_j^2\left[1-\cos(\eta_j)\frac{\theta_j^2}{12}\right]\quad \mbox{for some }0\leq \eta_j\leq \theta_j,$$
and
$$1-\frac{\pi^2}{12}\leq 1-\cos(\eta_j)\frac{\theta_j^2}{12}\leq 1+\frac{\pi^2}{12}.$$
Then, for $k $ large and odd, we get
\begin{equation*}\begin{split}
\sum_{j=2}^k {1\over [2(1-\cos \theta_j)]^{n-2 \over 2} }&= 2\sum_{j=2}^{\frac{k-1}{2}}\frac{1}{\theta_j^{n-2}}+2\sum_{j=2}^{\frac{k-1}{2}}\frac{1}{\theta_j^{n-2}}\frac{1-\left[1-\cos(\eta_j)\frac{\theta_j^2}{12}\right]^{\frac{n-2}{2}}}{\left[1-\cos(\eta_j)\frac{\theta_j^2}{12}\right]^{\frac{n-2}{2}}}\\
&=
\begin{cases}k^{n-2}   A_n \left(1+  O( \sigma_k ) \right) \quad
{\mbox {if}} \quad n \geq 4,\\
  k  \ln k  A_3 \left(1+ O(\sigma_k ) \right) \quad {\mbox {if}} \quad n =3,
  \end{cases}
\end{split}
\end{equation*}
where
$$A_n:=  {2\over (2\pi)^{n-2}} \, \sum_{j=1}^\infty j^{2-n},\;A_3 := \pi^{-1}\;\hbox{ and }\;
\sigma_k := \begin{cases} \begin{split}
k^{-2} & \quad  {\mbox {if}} \quad n>5, \\
k^{-2} \ln k & \quad  {\mbox {if}} \quad n=5, \\
k^{-1} & \quad {\mbox {if}} \quad n=4,\\
(\ln k)^{-1} & \quad {\mbox {if}} \quad n=3.
\end{split}
\end{cases}
$$
The case of $k$ even can be analogously treated, and therefore \eqref{uno} follows.

In the same spirit we also observe that, for $n\geq 4$ and $k$ odd, we have
\begin{equation*}\begin{split}
\sum_{j=1}^k {1\over |\oxi_1 -  \uxi_j|^{n-2}} &= {2\over R^{n-2} } \sum_{j=1}^{\frac{k-1}{2}} {1\over [4 \tau^2 + 2 (1-\tau^2) (1-\cos \theta_j ) ]^{n-2 \over 2} }\\
&= {2\over (2R)^{n-2} \tau^{n-2 }  } \left(\sum_{j=1}^{\frac{k-1}{2}}   {1\over [1  + \pi^2 {1-\tau^2 \over \tau^2} ({j-1 \over k} )^2]^{n-2 \over 2} } \right) (1+ O(\tau^2) ),
\end{split}
\end{equation*}
and for $n=3$,
\begin{equation*}\begin{split}
\sum_{j=1}^k {1\over |\oxi_1 -  \uxi_j|}
&= {1\over R \tau  } \left(\sum_{j=1}^{\frac{k-1}{2}}   {1\over [1  + \pi^2 {1-\tau^2 \over \tau^2} ({j-1 \over k} )^2]^{n-2 \over 2} } \right) (1+ O(|\ln \tau|^{-1}) ).
\end{split}
\end{equation*}
If $n \geq 4$, using  \eqref{par1} we get
\begin{equation}
\begin{split}\label{e3}
\sum_{j=1}^{\frac{k-1}{2}}   {1\over [1  + \pi^2 {1-\tau^2 \over \tau^2} ({j-1 \over k} )^2]^{n-2 \over 2} } &=  \left(\int_0^{\frac{k-3}{2}} {dx \over (1 + \pi^2 {1-\tau^2 \over \tau^2} ({x \over k} )^2)^{n-2 \over 2}  } \right) \, (1+ O ((k\tau )^{-2})) \\
&=  {k \tau \over   \pi \, \sqrt{1-\tau^2} } \left( \int_0^{\pi {\sqrt{1-\tau^2} \over \tau}  {{k-3} \over{2 k}} } {ds \over (1+ s^2)^{n-2 \over 2} }\right)  \, (1+ O ((k\tau )^{-2}) )\\
&= {k \tau \over   \pi \, \sqrt{1-\tau^2} }  \left( \int_0^{\infty} {ds \over (1+ s^2)^{n-2 \over 2} }\right)  \, (1+ O ((k\tau )^{-2})+ O(\tau^{n-3} ) ).
\end{split}
\end{equation}
If $n=3$, 
\begin{equation}
\begin{split}\label{e4}
\sum_{j=1}^{\frac{k-1}{2}}   &{1\over [1  + \pi^2 {1-\tau^2 \over \tau^2} ({j-1 \over k} )^2]^{1 \over 2} } =
{k \tau \over   \pi \, \sqrt{1-\tau^2} } \left( \int_0^{\pi {\sqrt{1-\tau^2} \over \tau}  {k-3 \over 2k} } {ds \over (1+ s^2)^{1 \over 2} }\right)  \, (1+ O ((k\tau )^{-2}) )\\
&\quad=  {k \tau \over   \pi \, \sqrt{1-\tau^2} } \left( \ln ( \pi {\sqrt{1-\tau^2} \over \tau}  {k-3 \over 2k} + \sqrt{1+  \pi^2 {1-\tau^2 \over \tau^2 }  \left({k-3 \over 2k}\right)^2 } ) \right)  \, (1+ O ((k\tau )^{-2}) )\\
&\quad = {k \tau \over   \pi \, \sqrt{1-\tau^2} } \ln \left( {\pi \over \tau} \right) \,  (1+ O(\tau^2 |\ln \tau |^{-1}  )).
\end{split}
\end{equation}
Combining \eqref{e3} and \eqref{e4} we obtain the validity of \eqref{due} for $k$ odd. The even case follows in the same way.

Analogous computations provide
\begin{equation}\begin{split} \label{tre}
\sum_{j=1}^k {1\over |\oxi_1 - \uxi_j|^{n}} =
\begin{cases}  C_n \, {k \over \tau^{n-1}} \, \left(1+  O( (\tau k)^{-2}) \right) \quad
{\mbox {if}} \quad n \geq 5,\\
C_n \, {k \over \tau^{n-1}} \, \left(1+  O( \tau^{n-1}) \right) \quad
{\mbox {if}} \quad n =3,4,
\end{cases}
\quad C_n := {2\over 2^{n} \pi}  \int_0^\infty {ds \over (1+ s^2)^{n \over 2}}.
\end{split}
\end{equation}

\section*{Appendix B: linear independence of the functions $z_j(y)$.}

We give here the proof of the linear independence of the functions $z_j(y)$, $j = 0,\ldots, 4n-3$, defined in Remark \ref{li}, in the case $k$ even. Indeed, we claim that, if there are  $4n-2$ constants $c_j$ with
\begin{equation}\label{unoli}
\sum_{j=0}^{4n-3} c_j z_j (y) = 0 , \quad \forall y \in \R^n,
\end{equation}
then $c_j = 0 $ for all $j$.	The result will follow from evaluating this expression in different points, properly chosen. Notice that, by definition, \eqref{unoli} can be written as
\begin{equation}\begin{split}\label{uno_bis}
&\left[c_0-2y_1c_{n+1}-2y_2c_{n+2}-2y_3c_{n+3}\right]z_0(y)+\left[c_1+|y|^2c_{n+1}-\sum_{\alpha=2}^ny_\alpha c_{n+\alpha+2}\right]z_1(y)\\
&\quad+\left[c_2+|y|^2c_{n+2}+c_{n+4}y_1-\sum_{\alpha=3}^nc_{2n+\alpha}y_\alpha\right]z_2(y)\\
&\quad+\left[c_3+|y|^2c_{n+3}+c_{n+5}y_1+c_{2n+3}y_2-\sum_{\alpha=4}^nc_{3n+\alpha-3}y_\alpha\right]z_3(y)\\
&\quad+\sum_{l=4}^n\left[c_l+c_{n+l+2}y_1+c_{2n+l}y_2+c_{3n+l-3}y_3\right]z_l(y)=0, \quad \forall y \in \R^n.
\end{split}\end{equation}
Since our solution $u $ satisfies
$$
u(y_1 ,  \ldots , y_j , \ldots, y_n) = u(y_1 ,  \ldots ,- y_j , \ldots, y_n), \quad \forall j=1, \ldots n,
$$
we have that necessarily, for every $j=1, \ldots n$,
$$
z_j (q) = 0 \quad \forall q= (q_1 , \ldots , q_n) \mbox{ with } q_j =0.
$$
Take a point of the form $q=(r, 0,...,0)$, with $r >0$. Thus, evaluating \eqref{uno_bis} at $q$ one gets\begin{equation}\label{xi1}
	c_0z_0(q)+c_1 z_1(q)+c_{n+1}(-2r z_0(q)+r^2z_1(q))=0.
\end{equation}
Given the decay of the functions $z_0$ and $z_1$, there are three constants $a$, $b$ and $c$ such that
\begin{equation}\label{du}
\begin{aligned}
	r^{n-2} z_0(q)&\sim -{n-2 \over 2} 2^{n-2 \over 2} a , \quad r^{n-1} z_1(q)\sim -{n-2 \over 2} 2^{n-2 \over 2} b ,\\
	r^{n-3} &\left(-2r z_0(q)+r^2z_1(q)\right) \sim -{n-2 \over 2} 2^{n-2 \over 2} c, \quad {\mbox {as}} \quad r \to \infty.
\end{aligned}
\end{equation}
The constants $a$, $b$ and $c$ can be computed explicitly
\begin{align*}
	a&= (1+ o_k (1) ) , \quad
	b= 2 (1+ o_k (1) ),\\
	c&= 4 k \lambda^{n-2 \over 2}  (1+ o_k (1) ),
\end{align*}
where $\lim_{k\to \infty} o_k (1) =0$. 
Evaluating \eqref{xi1} at three points $(r_1,0,...,0)$, $(r_2,0,...,0)$, $(r_3,0,...,0)$, with $r_i$ large, we arrive at a system which at main order looks like
$$\left(
\begin{array}{ccc}
	a r_1 & b & c r_1^2\\
	a r_2 & b & c r_2^2\\
	a r_3 & b & c r_3^3
\end{array}
\right)
\left(
\begin{array}{c} 
	c_0 \\ c_1 \\ c_{n+1}
\end{array}
\right)
=
\left(
\begin{array}{c} 
	0 \\ 0 \\ 0
\end{array}
\right).$$
Choosing properly $r_1\neq r_2\neq r_3$ it can be seen that the determinant of the matrix is not zero and thus necessarily
$$c_0=c_1=c_{n+1}=0.$$
Let us take now $q=(q_1, q_2,0,...,0)$. From \eqref{uno_bis} we obtain 
\begin{equation}\label{uu}c_2z_2(q)+c_{n+2}(-2q_2 z_0(q)+z_2(q)|q|^2)+c_{n+4}(-q_2z_1(q) + q_1 z_2 (q) )=0.
\end{equation}
If $q_2=0$, then $-q_2z_1(q) + q_1 z_2 (q) =0.$
Since $u (y)$ is invariant under rotations of angle ${2\pi \over k }$ in the $(y_1,y_2)$-plane, we have that $-q_2z_1(q) + q_1 z_2 (q) =0$ for all points of the form
$q= (e^{{2\pi \over k} (j-1) i} \tilde{r},0,0, 0 , ..)$, with $\tilde{r}=(r,0)$, for some  $j=2, \ldots , k$. Evaluating \eqref{uu} at these points $q$, we get 
$$
c_2z_2(q)+c_{n+2}(-2q_2 z_0(q)+z_2(q)|q|^2)=0.
$$
Replacing in the above equation the points $q_1= (e^{{2\pi \over k} (j-1) i} \tilde{r}_1,0,0, 0 , ..)$, $q_2= (e^{{2\pi \over k} (j-1) i} \tilde{r}_2,0,0, 0 , ..)$, with $\tilde{r}_1=(r_1,0)$, $\tilde{r}_2=(r_2,0)$, $r_1\not= r_2$, we get a $2\times 2$ system in $c_2$ and $c_{n+2}$. Arguing as before and choosing $r_1$, $r_2$ properly, we obtain  $c_2=c_{n+2}=0$. Thus so far we have proven that
$$
c_0=c_1=c_{n+1}=c_2=c_{n+2}=0.
$$


\medskip
Next we evaluate \eqref{uno_bis} at points of the form $q=(q_1,0,q_3,0,...,0)$ we deduce
$$c_3z_3(q)+c_{n+3}(-2q_3z_0(q)+|q|^2z_3(q))+c_{n+5}(q_1 z_3 (q) - q_3z_1(q))=0.$$
Taking now $\bar q = (-q_1,0,-q_3,0,...,0)$ and using that $z_0(\bar q)=z_0(q)$, $z_1(\bar q)= -z_1 (q)$, $z_3 (\bar q) = - z_3 (q)$, we obtain
$$c_3z_3(q)+c_{n+3}(-2q_3z_0(q)+|q|^2z_3(q))=0.$$
Similar decay rates as in \eqref{du} allow us to choose two different points of the form $q=(q_1,0,q_3,0,...,0)$ that inserted in the above equation produce an invertible $2\times 2$ system. This gives $c_3=c_{n+3}=0$.

\medskip
Evaluate \eqref{uno_bis} at $q=(q_1,q_2, 0,0,0..0)$ to get
$$
c_{n+4} [ q_2 z_1 (q) - q_2 z_2(q) ]=0.
$$
Since our solution $u$ is not radially symmetric in the $(y_1, y_2)$-plane, the function $y_2 z_1(y) - y_1 z_2 (y)$ is not identically zero. Thus we get $c_{n+4}=0$. Arguing similarly, we prove that $c_{n+5}=c_{2n+3}=0$ (we use \eqref{uno_bis} for $q=(q_1,0,q_3, 0,0,..)$ and $q=(0,q_2,q_3,0,0,...0)$ respectively, and the fact that the solution is not radially symmetric in the $(y_1,y_3)$ and $(y_2,y_3)$ planes).

\medskip
Let now $\ell =4, \ldots, n$ and evaluate \eqref{uno_bis} at points of the form $q=(q_1, 0, .., 0 , q_\ell , ...)$ We obtain
$$
c_\ell z_\ell (q) + c_{2n+\ell} (-q_\ell z_1 (q) + q_1 z_\ell (q)) =0.
$$
Since $u$ is not radially symmetric in the $(y_1,y_\ell)$-plane, the function 
$y \to -y_\ell z_1 (y) + y_1 z_\ell (y)$ is not identically zero. Choose now two different points of the form $q=(q_1,0,q_\ell,0,...,0)$ that inserted in the above equation produce an invertible $2\times 2$ system. This gives $c_\ell=c_{2n+\ell}=0$, for $\ell =4,\ldots,n$.

Hence, so far we have proven that
$$
c_0=c_1=.....=c_{3n}=0.
$$
It remains to see that
$$
\sum_{j=4}^n c_{3n+j-3} [ y_3 z_j(y) - y_j z_3 (y) ] = 0 \quad \forall y,
$$
implies $c_{3n+4}= \ldots = c_{4n-3} = 0$. Let us evaluate this linear combination at the points  $ q_{3,j} = (0,0, q_3, 0,..., q_j , 0, ...)$, for $j =4, \ldots , n$. We get 
$$
c_{3n+j -3 } (q_3 z_j (q_{3,j} ) - q_j z_j (q_{3,j} )) = 0.
$$
Since $u$ is not radially symmetric in the $(y_3, y_j)$-plane, we can choose $q_{3,j}$ to show that 
 $c_{3n+j - 3} = 0 $ for all $j=4,\ldots,n$. We then conclude that \eqref{uno_bis} implies that
 $$
 c_j=0, \quad {\mbox {for all}} \quad j=0,1, \ldots , 4n-3.
 $$


\begin{thebibliography}{1}
\setlinespacing{0.98}
\frenchspacing

\bibitem{Au}
T. Aubin,  \'Equations diff\'{e}rentielles non-lin\'{e}aires et probl\`{e}me de Yamabe concernant la
courbure scalaire, {\em J. Math. Pures Appl.} 55 (1976), 269-296.

\bibitem{CGS} L.A. Caffarelli, B. Gidas, J. Spruck, Asymptotic symmetry and local behavior
of semilinear elliptic equations with critical Sobolev growth, {\em Comm. Pure
Appl. Math.} 42 (1989), 271-297.

\bibitem{D} W. Ding, On a conformally invariant elliptic equation on $R^n$, {\em Communications in Mathematical Physics}  107 (1986), 331-335.


\bibitem{DJKM} T. Duyckaerts, H. Jia, C. Kenig, F. Merle,  Soliton resolution along a sequence of times for the focusing energy critical wave equation. {\em Geom. Funct. Anal.} 27 (2017), no. 4, 798--862.

\bibitem{DKM}{T. Duyckaerts, C. Kenig, F. Merle},  Solutions of the focusing nonradial critical wave equation with the compactness property, {\em Ann. Sc. Norm. Super. Pisa Cl. Sci.} Vol. XV (2016), 731-808.

\bibitem{DKM2} T. Duyckaerts, C. Kenig, F. Merle, Profiles of bounded radial solutions of the focusing, energy-critical wave equation. {\em Geom. Funct. Anal.} 22 (2012), no. 3, 639-698.

\bibitem{DKM3} T. Duyckaerts, C. Kenig, F. Merle, Universality of the blow-up profile for small type II blow-up  solutions of the energy-critical wave equation: the nonradial case. {\em J. Eur. Math. Soc. (JEMS)} 14 (2012), no. 5, 1389-1454.

\bibitem{dPMPP}M. del Pino, M. Musso, F. Pacard, A. Pistoia, Large energy entire solutions for the Yamabe equation. {\em Journal of Differential Equations } 251 (2011), 2568--2597.

\bibitem{dPMPP2}  M. del Pino, M. Musso, F. Pacard. A. Pistoia. Torus action on $S^n$ and sign changing solutions for conformally invariant equations. {\em Annali della Scuola Normale Superiore di Pisa } (5) 12 (2013), no. 1, 209--237.


\bibitem{K} N. Kapouleas.  Doubling and desingularization constructions for minimal surfaces. Surveys in geometric analysis and relativity,  {\em Adv. Lect. Math. (ALM)}, 20, Int. Press, Somerville, MA, (2011), 281-325.

    \bibitem{K1} N. Kapouleas. Minimal surfaces in the round three-sphere by doubling the equatorial two-sphere, I. {\em J. Differential Geom.} 106 (2017), no. 3, 393--49.

    \bibitem{K2} N. Kapouleas, P.  McGrath. Minimal surfaces in the round three-sphere by doubling the equatorial two-sphere, II. {\em Comm. Pure Appl. Math.} 72 (2019), no. 10, 2121--2195.

    \bibitem{KM1} C. Kenig,  F.Merle, 
    Global well-posedness, scattering and blow-up for the energy-critical, focusing,
    non-linear Schr\"odinger equation in the radial case, {\em Invent. Math.} 166 (2006), 645-675.

        \bibitem{KM2} C. Kenig, F. Merle, Global well-posedness, scattering and blow-up for the energy-critical focusing nonlinear wave equation, {\em Acta Math.} 201 (2008), 147-212.

        \bibitem{KST} J. Krieger, W. Schlag, D. Tataru, Slow blow-up solutions for the $H^1 (\R^3)$ critical focusing semilinear wave equation. {\em Duke Math. J.} 147 (2009), 1-53.

\bibitem{MMW} M. Medina, M. Musso, J. Wei, Desingularization of Clifford torus and nonradial solutions to Yamabe
problem with maximal rank. {\em Journal of Functional Analysis} 276, Iss. 8 (2019), 2470--2523.

\bibitem{MW} M. Musso, J. Wei, Nondegeneracy of Nonradial  Nodal Solutions to Yamabe
 Problem. {\em Communications in Mathematical Physics } 340, Issue 3, (2015), 1049--1107.

\bibitem{rey} O. Rey, The role of the Green's function in a nonlinear elliptic equation involving the critical Sobolev exponent.  {\em J.
Funct. Anal. } 89 (1990), no. 1, 1-52.

\bibitem{Ta}
G. Talenti, Best constants in Sobolev inequality, {\em Ann.
Mat. Pura Appl.} 110 (1976), 353--372.

\end{thebibliography}
\end{document}